\documentclass{article}
\usepackage{color}
\usepackage{amssymb,amsmath,amsthm,graphicx, color, amsfonts,tabularx}
\usepackage[english]{babel}
\usepackage{authblk,relsize}
\usepackage{afterpage}
\usepackage{hyperref}
\usepackage[export]{adjustbox}
\usepackage{geometry}
\def\supp{\mathop{\rm supp}\nolimits}

\newtheorem{theorem}{Theorem}[section]

\newtheorem{lemma}[theorem]{Lemma}

\newtheorem{corollary}[theorem]{Corollary}
\newtheorem{definition}[theorem]{Definition}
\newtheorem{remark}[theorem]{Remark}
\newtheorem{example}[theorem]{Example}

\newcommand{\N}{\mathbb{N}}
\newcommand{\R}{\mathbb{R}}

\renewenvironment{proof}[1][.]{%
\bigskip\noindent{\bf Proof#1 }}{%
\hfill$\blacksquare$\bigskip}

\newcommand{\mR}{\mathcal R}
\newcommand{\mS}{\mathcal S}

\newcommand{\mP}{\mathcal P}

\newcommand{\K}{\mathcal K}

\newcommand{\ve}{\varepsilon}
\newcommand{\on}{\operatorname}

\begin{document}
\pagestyle{myheadings}
\title{A multiresolution algorithm to approximate the Hutchinson measure for IFS and GIFS}

\author[1]{Rudnei D. da Cunha
\thanks{
Instituto de Matem\'atica e Estat\'istica - UFRGS, Av. Bento Gon\c calves 9500, 91500-900 Porto Alegre - RS - Brazil.\\
E-mail: rudnei.cunha@ufrgs.br}}
\author[1]{Elismar R. Oliveira
\thanks{E-mail: elismar.oliveira@ufrgs.br}}
\author[2]{Filip Strobin\thanks{
Instititute of Mathematics, Lodz University of Technology, W\'olcza\'nska 215, 90-924 {\L}\'od\'z, Poland.}}
\affil[1]{Universidade Federal do Rio Grande do Sul}
\affil[2]{Lodz University of Technology}

\date{\today}
\maketitle

\noindent\rule{\textwidth}{0.4mm}
\begin{abstract}
We introduce a discrete version of the Hutchinson--Barnsley theory providing algorithms to approximate the Hutchinson measure for iterated function systems (IFS) and generalized iterated function systems (GIFS), complementing the discrete version of the deterministic algorithm considered in our previous work.
\end{abstract}
\vspace {.8cm}



\section{Introduction}
This work is a sequel of \cite{DOS} where we approximate attractors for IFS and GIFS and their fuzzy versions by a discrete version of the deterministic algorithm.

Here we adapt the introduced theory to find a discrete version of the Markov operator acting on probability measures. Once we prove that this operator satisfies the assumptions of the discrete version of the Banach fixed point theorem (Theorem~\ref{dfp}), we justify the existence of a discrete measure approximating the Hutchinson Measure associated to the underlying IFS or GIFS with probabilities. This fits to problems addressed by several authors such as \cite{Ob1}, \cite{Ob2}, \cite{Bla}, \cite{Fro}, \cite{GN}, \cite{GMN} and \cite{CJ}. See the introduction of \cite{GMN} for a detailed review on the advances of most of these references.

 In recent years some authors have made significant advances in the task of giving an approximation of the invariant measure of a dynamical system with precisely estimated error. For example, the
strategy in \cite{GMN} is to replace the original Markov (transfer) operator, let us say $M$, by a finite rank one, let us say $\tilde{M}$. Then the authors use some powerful techniques and a functional analytic approach to approximate the actual invariant measure $\mu$ (i.e. $M(\mu)=\mu$) by the fixed point $\tilde{\mu}$ of $\tilde{M}$. Our approach is different: modifying the original IFS $\mS$ by replacing its maps with their certain discretizations defined on appropriately dense grids, we obtain a new IFS $\hat{S}$, whose Markov operator $\hat{M}$  need not be contractive (see Example \ref{exxx}), but its iterations give approximations of the invariant measure of $\mS$.
Not only we present an algorithm that approximates the invariant measure for an IFS with probabilities but,  similarly as in \cite{GMN}, we provide an explicit control on the convergence and the error estimation. The main novelty is the fact that we can apply this procedure even for generalized iterated function systems which have no dynamical counterpart, presenting, as far as we know, the first rigorous computation of the invariant measures in the literature, including place depending probabilities. The advantages of our approach (discretization of the space and of the associated operators)
 is that we obtain a very simple algorithm to be coded exhibiting an outstanding performance for computation without the need of a sophisticated computational machinery. Moreover, the lack of requirements allow us to extrapolate the application of the algorithm to investigate (heuristically) other situations where our hypothesis fail, for instance as in Example~\ref{ex_conze} where the convergence can not happen since we have more than one invariant measures.

 In Example~\ref{ex:galatolo_final} we make a comparative analysis  between our algorithm and the method employed in \cite{GMN} for a contracting case. Also, in Section~\ref{sec:Further applications}, subsection~\ref{sec:approx_integrals}, we make a comparison between our computations and with the computational results given in a recent preprint~\cite{CJ}. It is worth to mention that \cite{CJ} follows a very different approach to the problem: the maps are also assumed to be Lipschitz contractions but the authors use an operator theoretic approach similar to that given in \cite{JP} where the family of maps must be \emph{complex contracting} leading to a good numerical results.

The paper is organized as follows:\\
In Section~\ref{sec:basics hutc barnsley} we recall some basic facts on the Hutchinson--Barnsley theory and its certain generalization due to Miculescu and Mihail. Not only we formulate the Hutchinson-Barnsley theorem on the existence of attractors, but also, we recall definition of (generalized) Markov operator and provide theorems on the existence the Hutchinson measure.

After, in Section~\ref{sec:discret fix point thm},  we present the theory developed in \cite{DOS}. We formulate a discrete version of the Banach fixed point theory and present its application -  discretizations of the Hutchinson-Barnsley theorem, as well as its generalization due to Miculescu and Mihail.

In Section 4 we obtain the main result of our paper, which shows that we can approximate the Hutchinson measure of given IFS (or GIFS) by appropriate discrete measure, obtained by iterating process of discrete Markov operator.

In Section~\ref{sec:discret determ algor for hutch} we introduce an algorithm that can be used to generate discrete Hutchinson measures (for IFSs and GIFSs) and also we present some examples to illustrate it. Note that our algorithm draw the pictures of approximations of the attractors both of IFSs and GIFSs. Other such algorithms for GIFSs can be found in \cite{JMS} and \cite{MM3}.

Finally in Section~\ref{sec:Further applications} we address the problem of estimating the integral of functions with respect to the Hutchinson measures from \cite{CJ}, the problem of the Projected Hutchinson measures from \cite{EO} and the problem of IFS/GIFS with place dependent probabilities from \cite{Mi}.

Note that some results we recall in the initial sections (for example, Lemmas \ref{lem3} and \ref{lem3g}, and Sections 3.2 and 3.3) are given just for the sake of completeness and for the reader's convenience.

\section{Basics of the Hutchinson--Barnsley theory and its certain generalization}\label{sec:basics hutc barnsley}
\subsection{Iterated function systems and the Hutchinson--Barnsley theorem}
Let $(X,d)$ be a  metric space. We say that $f:X\to X$ is a \emph{Banach contraction}, if the Lipschitz constant $\on{Lip}(f)<1$.
\begin{definition}\emph{
An }iterated function system\emph{ (IFS in short)  $\mS=(X,(\phi_j)_{j=1}^{L})$ consists of a finite family  $\phi_1,...,\phi_L$ of continuous selfmaps of $X$. Each IFS $\mS$ generates the map $F_\mS:\K(X)\to\K(X)$ (where $\K(X)$ denotes the family of all nonempty and compact subsets of $X$), called }the Hutchinson operator\emph{, defined by
$$
\forall_{K\in\K(X)}\;F_\mS(K):=\bigcup_{j=1}^L\phi_j(K).
$$
  By an \emph{attractor} of an IFS $\mS$ we mean a (necessarily unique) set} $A_\mS\in \K(X)$ which satisfies
$$
A_\mS=F_\mS(A_\mS)=\bigcup_{j=1}^L\phi_j(A_\mS)
$$
and such that for every $K\in\K(X)$, the sequence of iterates $(F_\mS^k(K))_{k=0}^{\infty}$ converges to $A_\mS$ with respect to the Hausdorff-Pompeiu metric $h$ on $\K(X)$(see \cite{BM} for details on the Hausdorff-Pompeiu metric).
\end{definition}
The classical Hutchinson--Barnsley theorem \cite{Bar}, \cite{Hut} states that:
\begin{theorem}\label{H-B}
Each IFS $\mS$ consisting of Banach contractions on a complete metric space $X$ admits attractor.
\end{theorem}
This result can be proved with a help of the Banach fixed point theorem as it turns out that $F_\mS$ is a Banach contraction provided each $\phi_j$ is a Banach contraction   (see, for example, \cite[Section 3.2]{Hut}):
\begin{lemma}\label{lem3}
Let $(X,d)$ be a metric space and $\mS=(X,(\phi_j)_{j=1}^{L})$  be an IFS consisting of Banach contractions. Then $F_\mS$ is a Banach contraction and $\on{Lip}(F_\mS)\leq\max\{\on{Lip}(\phi_j):j=1,...,L\}$.
\end{lemma}

Given an IFS $\mS=(X,(\phi_j)_{j=1}^L)$ consisting of Banach contractions, we will denote $$\alpha_\mS:=\max\{\on{Lip}(\phi_1),...,\on{Lip}(\phi_L)\}.$$

\subsection{IFSs with probabilities and the Hutchinson measure}
For a metric space $(X,d)$, by $\mP(X)$ we denote the family of all nonnegative Borel probability measures $\mu$ with compact support, that is, for which there exists a compact set $K$ so that $\mu(X\setminus K)=0$. For every $\mu,\nu\in\mP(X)$, define
\begin{equation}
d_{MK}(\mu, \nu)=\sup \left\{\left|\int_{X} f \mathrm{d} \mu-\int_{X} f \mathrm{d} \nu\right| : f \in \operatorname{Lip}_{1}(X, \mathbb{R})\right\},
\end{equation}
where $\operatorname{Lip}_{1}(X, \mathbb{R})$ is the set of maps $f:X\to\R$ with $\on{Lip}(f)\leq 1$. It is known (see \cite{Bog},\cite{KTMV},\cite{Hut}, \cite{Bar}) that $d_{MK}$ is a metric, and it is complete provided $(X,d)$ is complete. The metric $d_{MK}$ is called the \emph{Monge-Kantorovich metric} or the \emph{Hutchinson metric}. \\
For a Borel measurable map $w:X\to Y$ (where $Y$ is a metric space) and $\mu\in\mP(X)$, by $w^\sharp\mu$ we denote the measure on $Y$ defined by
\begin{equation}
w^\sharp\mu(B):=\mu(w^{-1}(B)).
\end{equation}
It is known that $w^\sharp\mu\in\mP(Y)$ and that for any continuous map $f:Y\to\R$ and every $\mu\in\mP(X)$,
\begin{equation}\label{filip1}
\int_{X}f \;dw^\sharp\mu=\int_Xf\circ w\;d\mu.
\end{equation}

\begin{definition}
\emph{By an }IFS with probabilities (IFSp in short) \emph{we mean a triple $\mS=(X,(\phi_j)_{j=1}^L,(p_j)_{j=1}^L)$ so that $(X,(\phi_j)_{j=1}^L)$ is an IFS and $p_1,...,p_L>0$ with $\sum_{j=1}^Lp_j=1$.\\
Each IFSp generates the map $M_\mS:\mP(X)\to\mP(X)$, called the} Markov operator, \emph{which adjust to every $\mu\in\mP(X)$, the measure $M_\mS(\mu)$ defined by:
\begin{equation*}
M_{\mathcal{S}}(\mu)(B)=\sum_{j=1}^{L} p_{j} \mu\left(\phi_{j}^{-1}(B)\right)=\sum_{j=1}^Lp_j\phi_j^\sharp\mu, \quad
\end{equation*}
for any Borel set $B\subset X$.\\
By  a Hutchinson measure \emph{of an IFSp $\mS$ we mean  a (necessarily unique)} measure $\mu_\mS\in\mP(X)$ which satisfies
\begin{equation}
\mu_\mS=M_\mS(\mu_\mS)
\end{equation}
and such that for every $\mu\in\mP(X)$, the sequence of iterates $M^k_\mS(\mu)$ converges to $\mu_\mS$ with respect to $d_{MK}$.}
\end{definition}
Observe that by (\ref{filip1}), for every IFSp $\mS$ and every continuous map $f:X\to\R$,
 \begin{equation}
\int_{X}f \;dM_\mS(\mu)=\sum_{j=1}^Lp_j\int_Xf\circ \phi_j\;d\mu.
\end{equation}
In fact, the above formula characterizes the Markov operator $M_\mS$.
The following result is known (see mentioned papers,   for example \cite[Section 4.4]{Hut}).
\begin{theorem}
Each   IFSp on a complete metric space consisting of Banach contractions admits the Hutchinson measure.
\end{theorem}
This result can be proved with a help of the Banach fixed point theorem as the following lemma holds:
\begin{lemma}\label{rnf}
If an IFSp $\mS$ consists of Banach contractions, then $M_\mS$ is a Banach contraction and $\on{Lip}(M_\mS)\leq\alpha_{\mS}$.
\end{lemma}
For a measure $\mu\in\mP(X)$, by its support we mean the set
$$
\on{supp}(\mu)=\{x\in X:\mu(U)>0\;\mbox{for all open sets}\;U\ni x\}.
$$
Alternatively, $\on{supp}(\mu)$ is the largest compact set $C$ such that $\mu(U)>0$ for every open set $U$ with $U\cap C\neq \emptyset$.
The following result shows that there is a strong relationship between the Markov and Hutchinson operators. It implies that the support of the Hutchinson measure is exactly the attractor for a given IFS  (see \cite[Section 4.4, Thm. 4]{Hut}).
\begin{lemma}\label{filip5}
In the above frame, for every $\mu\in\mP(X)$, $\on{supp}(M_\mS(\mu))=F_\mS(\on{supp}(\mu))$.
\end{lemma}

\subsection{Generalized IFSs in the sense of Miculescu and Mihail}
Here we recall some basics of a generalization of the classical IFS theory introduced by R. Miculescu and A. Mihail in 2008. For references, see \cite{M1}, \cite{MM}, \cite{MM1}, \cite{SS1} and references therein.\\
If $(X,d)$ is a metric space and $m\in\N$, then by $X^m$ we denote the Cartesian product of $m$ copies of $X$. We consider it as a metric space with the maximum metric
$$
d_m((x_0,...,x_{m-1}),(y_0,...,y_{m-1})):=\max\{d(x_0,y_0),...,d(x_{m-1},y_{m-1})\}.
$$
A map $f:X^m\to X$ is called a \emph{generalized Banach contraction}, if $\on{Lip}(f)<1$.

It turns out that a counterpart of the Banach fixed point theorem holds. Namely,
if $f:X^m\to X$ is a generalized Banach contraction, then there is a unique point $x_*\in X$ (called \emph{a generalized fixed point} of $f$), such that $f(x_*,...,x_*)=x_*$. Moreover, for every $x_0,...,x_{m-1}\in X$, the sequence $(x_k)$ defined by $$x_{k+m}=f(x_k,...,x_{k+m-1}),\;\;k\geq 0,$$
converges to $x_*$.
\begin{remark}\label{filip2}\emph{
It is worth to observe that if $f:X^m\to X$ is a generalized Banach contraction, then the map $\tilde{f}:X\to X$ defined by $\tilde{f}(x):=f(x,...,x)$, is a Banach contraction and $\on{Lip}(\tilde{f})\leq \on{Lip}(f)$.}
\end{remark}

\begin{definition}\emph{
A }generalized iterated function system of order $m$\emph{ (GIFS in short) $\mS=(X,(\phi_j)_{j=1}^{L})$ consists of a finite family  $\phi_1,...,\phi_L$ of continuous maps from $X^m$ to $X$. Each GIFS $\mS$ generates the map $F_\mS:\K(X)^m\to\K(X)$, called }the generalized Hutchinson operator \emph{, defined by
$$
\forall_{K_0,...,K_{m-1}\in \K(X)}\;F_\mS(K_0,...,K_{m-1}):=
\bigcup_{j=1}^L\phi_j(K_0\times...\times K_{m-1}).
$$
By  an \emph{attractor} of a GIFS $\mS$ we mean a (necessarily unique) set $A_\mS\in\K(X)$ which satisfies
$$
A_\mS=F_\mS(A_\mS,...,A_\mS)=\bigcup_{j=1}^L\phi_j(A_\mS\times...\times A_\mS)
$$
and such that for every $K_0,...,K_m\in\K(X)$, the sequence $(K_k)$ defined by
$$
K_{k+m}:=F_\mS(K_k,...,K_{k+m-1}),\;\;k\geq 0,
$$
converges to $A_\mS$.
}\end{definition}
 \begin{theorem}
Each GIFS $\mS$ on a complete metric space consisting of generalized Banach contractions generates the (unique) attractor $A_\mS$.
\end{theorem}
This result can be proved with a help of a certain generalization of the Banach fixed point theorem as it turns out that $F_\mS$ is a Banach contraction provided each $\phi_j$ is a Banach contraction (see, for example, \cite[Lemma 3.6]{MM}):

\begin{lemma}\label{lem3g}
Let $(X,d)$ be a metric space and $\mS=(X,(\phi_j)_{j=1}^{L})$ be a GIFS consisting of generalized Banach contractions. Then $F_\mS$ is a generalized Banach contraction with $\on{Lip}(F_\mS)\leq\alpha_\mS:=\max\{\on{Lip}(\phi_j):j=1,...,L\}$.
\end{lemma}
In \cite{M2}   and \cite{MM2}, Miculescu and Mihail studied the counterpart of the Hutchinson measure for GIFSs.
\begin{definition}
\emph{By a }GIFS with probabilities (GIFSp in short) \emph{we mean a triple $\mS=(X,(\phi_j)_{j=1}^L,(p_j)_{j=1}^L)$ so that $(X,(\phi_j)_{j=1}^L)$ is a GIFS and $p_1,...,p_L>0$ with $\sum_{j=1}^Lp_j=1$.\\
Each GIFSp generates the map $M_\mS:\mP(X)^m\to\mP(X)$, called the} generalized Markov operator \emph{which adjust to any $\mu_0,...,\mu_{m-1}\in\mP(X)$, the measure $M_\mS(\mu_0,...,\mu_{m-1})$ defined by
\begin{equation*}
M_{\mathcal{S}}(\mu_0,...,\mu_{m-1})(B):=\sum_{j=1}^{L} p_{j} (\mu\times...\times \mu_{m-1})(\phi_j^{-1}(B))=\sum_{j=1}^Lp_j\phi_j^\sharp(\mu_0\times...\times\mu_{m-1})(B), \quad
\end{equation*}
for any Borel set $B\subset X$, and where $\mu_0\times...\times\mu_{m-1}$ is the product measure.\\
By the } generalized Hutchinson measure \emph{of a GIFSp $\mS$ we mean the unique measure $\mu_\mS\in\mP(X)$ which satisfies
\begin{equation}
\mu_\mS=M_\mS(\mu_\mS,...,\mu_\mS)
\end{equation}
and such that for every $\mu_0,...,\mu_{m-1}\in\mP(X)$, the sequence $(\mu_k)$ defined by $\mu_{m+k}:=M_\mS(\mu_{k},...,\mu_{k+m-1})$, converges to $\mu_\mS$ with respect to $d_{MK}$.}
\end{definition}
Once again we observe that by (\ref{filip1}), for every IFSp $\mS$ and every continuous map $f:X\to\R$,
 \begin{equation}\label{filipaaa1}
\int_{X}f  \;dM_\mS(\mu_0,...,\mu_{m-1})=\sum_{j=1}^Lp_j\int_{X^m}f\circ \phi_j\;d(\mu_0\times...\times \mu_{m-1}).
\end{equation}
 For $m\in\N$ and $a_0,...,a_{m-1}\geq 0$, we say that $f:X^m\to X$ is a $(a_0,...,a_{m-1})$-contraction, if
\begin{equation}
\forall_{(x_0,...,x_{m-1}),(y_0,...,y_{m-1})\in X^m}\;d(f(x_0,...,x_{m-1}),f(y_0,...,y_{m-1}))\leq \sum_{j=0}^{m-1}a_jd(x_j,y_j)
\end{equation}
It is easy to see that the Lipschitz constant of $(a_0,...,a_{m-1})$-contraction  $f$ is bounded by $\on{Lip}(f)\leq\sum_{j=0}^{m-1}a_j$. In particular, if $\sum_{j=0}^{m-1}a_j<1$, then each $(a_0,...,a_{m-1})$-contraction is a generalized Banach contraction.\\
 Miculescu and Mihail in \cite{M2} and \cite{MM2} proved the following theorem (in fact,   they assumed $m=2$, but, as we will see, the proof works for general case):
\begin{theorem}\label{newfilipp1}
Assume that $\mS$ is a GIFSp on a complete metric space consisting of $(a_0,...,a_{m-1})$-contractions, where $\sum_{j=0}^{m-1}a_j<1$. Then $\mS$ admits the Hutchinson measure.
\end{theorem}
The proof bases on the following  lemma, which generalizes Lemma \ref{rnf}. As \cite[Theorem 3.4 ]{MM2} is stated for $m=2$, for completeness, we give here a short proof for the general case:
\begin{lemma}
If a GIFSp $\mS$ consists of $(a_0,...,a_{m-1})$-contractions, then $M_\mS$ is also $(a_0,...,a_{m-1})$-contraction.
\end{lemma}

\begin{proof}
At first, by induction we show that for every $m\geq 1$, every $(a_0,...,a_{m-1})$-contraction $\phi:X^m\to X$, nonexpansive map $f:X\to\R$ and measures $\mu_0,...,\mu_{m-1},\nu_0,...,\nu_{m-1}\in\mP(X)$, we have
\begin{equation}\label{newfilip}
\left|\int_{X^m}f\circ \phi\;d(\mu_0\times...\times\mu_{m-1})-\int_{X^m}f\circ \phi\;d(\nu_0\times...\times\nu_{m-1})\right|\leq a_0d_{MK}(\mu_0,\nu_0)+...+a_{m-1}d_{MK}(\mu_{m-1},\nu_{m-1})
\end{equation}
The case $m=2$ is proven in the proof of \cite[Theorem 3.4]{MM2}, and the case $m=1$ is a folklore (the proof can be found in \cite{KTMV}; it also can be deduced from the case $m=1$). Assume that (\ref{newfilip}) holds for some $m$. We will prove it for $m+1$. For every $(a_0,...,a_{m})$-contraction $\phi:X^m\to X$, nonexpansive map $f:X\to\R$ and measures $\mu_,...,\mu_{m},\nu_0,...,\nu_{m}\in\mP(X)$, we have
\begin{equation*}
\left|\int_{X^{m+1}}f\circ \phi\;d(\mu_0\times...\times\mu_m)-\int_{X^{m+1}}f\circ \phi\;d(\nu_0\times...\times\nu_m)\right|\leq
\end{equation*}
\begin{equation*}
\left|\int_{X^{m+1}}f\circ \phi\;d(\mu_0\times...\times\mu_m)-\int_{X^{m+1}}f\circ \phi\;d(\mu_0\times...\times\mu_{m-1}\times\nu_m)\right|+
\end{equation*}
\begin{equation*}
+\left|\int_{X^{m+1}}f\circ \phi\;d(\mu_0\times...\times\mu_{m-1}\times\nu_m)-\int_{X^{m+1}}f\circ \phi\;d(\nu_0\times...\times\nu_m)\right|\leq
\end{equation*}
\begin{equation*}
\left|\int_{X^{m}}\left(\int_Xf\circ \phi\;d\mu_m\right)\;d(\mu_0\times...\times\mu_{m-1})-\int_{X^{m}}\left(\int_Xf\circ \phi\;d\nu_m\right)\;d(\mu_0\times...\times\mu_{m-1})\right|+
\end{equation*}
\begin{equation*}
\left|\int_{X}\left(\int_{X^m}f\circ \phi\;d(\mu_0\times...\times\mu_{m-1})\right)\;d\nu_m+\int_{X}\left(\int_{X^m}f\circ \phi\;d(\nu_0\times...\times\nu_{m-1})\right)\;d\nu_m\right|\leq
\end{equation*}
\begin{equation*}
\leq \int_{X^{m}}\left|\int_Xf\circ \phi\;d\mu_m - \int_Xf\circ \phi\;d\nu_m\right|\;d(\mu_0\times...\times\mu_{m-1})  +
\end{equation*}
\begin{equation*}+
\int_X\left|\int_{X^m}f\circ \phi\;d(\mu_0\times...\times\mu_{m-1})-\int_{X^m}f\circ \phi\;d(\nu_0\times...\times\nu_{m-1})\right|d\nu_m\leq\;\bigoplus
\end{equation*}
we use the case $m=1$ for maps of the form $y\to \phi(x_0,...,x_{m-1},y)$, and the inductive assumption for maps  of the form $(y_0,...,y_{m-1})\to \phi(y_0,...,y_{m-1},x_m)$
\begin{equation*}
\bigoplus\;\leq\int_{X^m}a_md_{MK}(\mu_m,\nu_m)\;d(\mu_0\times...\times\mu_{m-1})+\int_X(a_0d_{MK}(\mu_0,\nu_0)+...+a_{m-1}d_{MK}(\mu_{m-1},\nu_{m-1})\;d\nu_m=
\end{equation*}
\begin{equation*}
 =a_0d_{MK}(\mu_0,\nu_0)+...+a_md_{MK}(\mu_m,\nu_m).
\end{equation*}
Thus the inductive step is finished.\\
Finally, by (\ref{filipaaa1}), we have for every $m\in\N$, GIFS $\mS$, nonexpansive map $f:X\to \R$, and measures $\mu_0,...,\mu_{m-1},\nu_0,...,\nu_{m-1}\in\mP(X)$,
\begin{equation*}
\left|\int_Xf\;dM_{\mS}(\mu_0,...,\mu_{m-1})-\int_Xf\;dM_{\mS}(\nu_0,...,\nu_{m-1})\right|=
\end{equation*}
\begin{equation*}
=\left|\sum_{j=1}^Lp_j\int_{X^m}f\circ \phi_j\;d(\mu_0\times...\times\mu_{m-1})-\sum_{j=1}^Lp_j\int_{X^m}f\circ \phi_j\;d(\nu_0\times...\times\nu_{m-1})\right|\leq
\end{equation*}
\begin{equation*}
\leq\sum_{j=1}^Lp_j\left|\int_{X^m}f\circ \phi_j\;d(\mu_0\times...\times\mu_{m-1})-p_j\int_{X^m}f\circ \phi_j\;d(\nu_0\times...\times\nu_{m-1})\right|\leq
\end{equation*}
\begin{equation*}
\leq\sum_{j=1}^Lp_j(a_0d_{MK}(\mu_0,\nu_0)+...+a_{m-1}d_{MK}(\mu_{m-1},\nu_{m-1}))=a_0d_{MK}(\mu_0,\nu_0)+...+a_{m-1}d_{MK}(\mu_{m-1},\nu_{m-1})
\end{equation*}
Hence, as $f$ was arbitrary, we get
\begin{equation*}
d_{MK}(M_{\mS}(\mu_0,...,\mu_{m-1}),M_{\mS}(\nu_0,...,\nu_{m-1}))\leq a_0d_{MK}(\mu_0,\nu_0)+...+a_{m-1}d_{MK}(\mu_{m-1},\nu_{m-1})
\end{equation*}
and the result follows.
\end{proof}

From the perspective of our future results, it is worth to note the following:
\begin{corollary}
If a GIFSp $\mS$ consists of $(a_0,...,a_{m-1})$-contractions so that $\sum_{i=0}^{m-1}a_i<1$, then the map $\overline{M}_\mS:\mP(X)\to\mP(X)$ defined by
$$
\forall_{\mu\in\mP(X)}\;\overline{M}_\mS(\mu):=M_\mS(\mu,...,\mu)
$$
is a Banach contraction with the Lipschitz constant $\on{Lip}(\overline{M}_\mS)\leq \sum_{i=0}^{m-1}a_i$.
\end{corollary}
\begin{proof}
  For every $\mu,\nu\in\mP(X)$, we have
\begin{equation*}
d_{MK}(\overline{M}_\mS(\mu),\overline{M}_\mS(\nu))=d_{MK}(M_\mS(\mu,...,\mu),M_{\mS}(\nu,...,\nu))\leq
\end{equation*}
\begin{equation*}\leq
a_0d_{MK}(\mu,\nu)+...+a_{m-1}d_{MK}(\mu,\nu)=(\sum_{i=0}^{m-1}a_i)d_{MK}(\mu,\nu)
\end{equation*}
\end{proof}

Now we give an extension of  Lemma \ref{filip5} for GIFSs:
\begin{lemma}\label{filip5a}
In the above frame, for every $\mu\in\mP(X)$, $\on{supp}(\overline{M}_\mS(\mu))=\overline{F}_\mS(\on{supp}(\mu))$ , where
$$
\forall_{K\in \K(X)}\;\overline{F}_\mS(K):=F_\mS(K,...,K).
$$
\end{lemma}
\begin{proof}
Assume first that $x\notin \on{supp}(\overline{M}_\mS(\mu))$. Then there exists an open set $U\ni x$ with $\overline{M}_\mS(\mu)(U)=0$, which means that for every $j=1,...,L$, $(\mu\times...\times \mu)(\phi_j^{-1}(U))=0$. Suppose on the contrary that $x\in \overline{F}_\mS(\on{supp}(\mu))$. Then there is $(x_0,...,x_{m-1})\in\on{supp}(\mu)\times...\times\on{supp}(\mu)$ and $j_0=1,...,L$ such that $\phi_{j_0}(x_0,...,x_{m-1})=x$. In particular, $(x_0,...,x_{m-1})\in\phi_{j_0}^{-1}(U)$, so there are open sets $W_0,...,W_{m-1}$ so that $(x_0,...,x_{m-1})\in W_0\times...\times W_{m-1}\subset \phi_{j_0}^{-1}(U)$. Hence
$$
0=(\mu\times...\times \mu)(\phi_{j_0}^{-1}(U))\geq (\mu\times...\times \mu)(W_0\times...\times W_{m-1})=\mu(W_0)\cdot...\cdot \mu(W_{m-1})>0,
$$
where the last inequality follows from the fact that $x_i\in\on{supp}(\mu)$ for every $i=0,...,m-1$. All in all, $x\notin \overline{F}_\mS(\on{supp}(\mu))$.\\
Now assume that $x\notin\overline{F}_\mS(\on{supp}(\mu))$. Since the latter set is closed, we can find an open set $U\ni x$ disjoint with $\overline{F}_\mS(\on{supp}(\mu))$. Then for any $j=1,...,L$, we have that $\phi_j^{-1}(U)\cap(\on{supp}(\mu)\times...\times\on{supp}(\mu))=\emptyset$. Hence, as $\on{supp}(\mu\times...\times\mu)=\on{supp}(\mu)\times...\times\on{supp}(\mu)$, we have that
$$
\overline{M}_\mS(\mu)(U)=\sum_{j=1}^Lp_j(\mu\times...\times\mu)(\phi^{-1}(U))=0.
$$
\end{proof}

Finally, let us remark that if $m=1$, then presented theory reduces to recalled earlier classical IFS theory. However, as was proved in \cite{MS} and \cite{S}, there are sets (even subsets of the real line) which are GIFSs attractors but which are not IFS attractors. Hence the theory of GIFSs is essentially wider than the IFSs' one.

\section{Discretization of the Banach fixed point theorem and its application to IFS theory} \label{sec:discret fix point thm}
\subsection{Discretization of the Banach fixed point theorem}
In this section we recall an important result from our earlier paper \cite{DOS}.
\begin{definition}\emph{
A subset ${\hat{X}}$ of a metric space $(X,d)$ is called an }$\ve$-net\emph{ of $X$, if for every $x\in X$, there is $y\in {\hat{X}}$ such that $d(x,y)\leq \ve$.\\
A map $r:X\to {\hat{X}}$ such that $r(x)=x$ for $x\in {\hat{X}}$ and $d(x,r(x))\leq\ve$ for all $x\in X$ will be called an }$\ve$-projection\emph{ of $X$ to $\hat{X}$.\\
For $f:X\to X$, by its} $r$-discretization\emph{ we will call the map $\hat{f}:=(r\circ f)_{\vert \hat{X}}$.\\
Similarly, if $f:X^m\to X$, by its} $r$-discretization\emph{ we will call the map $\hat{f}:=(r\circ f)_{\vert \hat{X}^m}$.\\
}
\end{definition}
The following result can be considered as a discrete version of the Banach fixed point theorem.

\begin{theorem}\label{dfp} (\cite[Theorem 4.2]{DOS})
Assume that $(X,d)$ is a complete metric space and $f:X\to X$ is a Banach contraction with the unique fixed point $x_*$ and the Lipschitz constant $\alpha$. Let $\ve>0$, ${\hat{X}}$ be an $\ve$-net, $r:X\to{\hat{X}}$ be an $\ve$-projection and $\hat{f}$ be an $r$-discretization of $f$.\\
For every $x\in {\hat{X}}$ and $n\in\N$,
\begin{equation}\label{approx}
d(\hat{f}^n(x),x_*)\leq\frac{\ve}{1-\alpha}+\alpha^nd(x,x_*).
\end{equation}
  In particular, there exists a point $y\in X$ so that $d(x_*,y)\leq \frac{2\ve}{1-\on{Lip}(f)}$ and which can be reached as an appropriate iteration of $\hat{f}$ of an arbitrary point of $\hat{X}$.
\end{theorem}

\subsection{Discretization of hyperspace $\K(X)$}

\begin{definition}\emph{
We say that an $\ve$-net $\hat{X}$ of a metric space $X$ is }proper\emph{, if for every bounded $D\subset X$, the set $D\cap\hat{X}$ is finite.}
\end{definition}
Note that proper $\ve$-nets are discrete (as topological subspaces), but the converse need not be true.
The existence of proper $\ve$-nets for every $\ve>0$ is guaranteed by the assumption that $X$ has so-called \emph{Heine--Borel property}, that is, the assumption that each closed and bounded set is compact. In particular, Euclidean spaces and compact spaces admit such nets.

\begin{lemma}\label{filip4}  (\cite[Lemma 5.2]{DOS})
Assume that $(X,d)$ is a metric space, ${\hat{X}}$ is a proper $\ve$-net and $r:X\to\hat{X}$ is an $\ve$-projection . Then:
\begin{itemize}
\item[(i)] $\K(\hat{X})$ consists of all finite subsets of ${\hat{X}}$;
\item[(ii)] $\K(\hat{X})$ is an $\ve$-net of $\K(X)$;
\item[(ii)] the map $r:\K(X)\to \K(\hat{X})$ defined by $r(K):=\{r(x):x\in K\}$, is an $\ve$-projection of $\K(X)\to\K(\hat{X})$ (using the same letter $r$ will not lead to any confusion).
\end{itemize}
\end{lemma}
\subsection{Discretization of the Hutchinson-Barnsley theorem and its version for GIFSs}
Assume that $\mS=(X,(\phi_j)_{j=1}^L)$ is a GIFS.  Recall that $\overline{F}_\mS:\K(X)\to\K(X)$  is given by
\begin{equation}
\forall_{K\in\K(X)}\;\overline{F}_\mS(K):=F_\mS(K,...,K)=\bigcup_{j=1}^L\phi_j(K\times...\times K).
\end{equation}
\begin{remark}\label{filip3}\emph{\\
(1) If $m=1$, then $\mS$ is an IFS and $\overline{F}_\mS=F_\mS$, the Hutchinson operator.\\
(2) If $\mS$ consists of generalized Banach contractions and $(X,d)$ is complete, then
by Remark \ref{filip2} and Lemma \ref{lem3g}, we see $\overline{F}_\mS$ is a Banach contraction, $\on{Lip}(\overline{F}_\mS)\leq\alpha_\mS$ and the unique fixed point of $\overline{F}_\mS$ equals the generalized attractor $A_\mS$.}
\end{remark}
For $\delta>0$, a set $A_\delta\in\K(X)$ will be called \emph{an attractor of $\mS$ with resolution $\delta$}, if $h(A_\delta,A_\mS)\leq\delta$.
\begin{theorem}\label{tt2}  (\cite[Theorem 6.2]{DOS})
Let $(X,d)$ be a complete metric space and $\mS$ be a GIFS on $X$ consisting of generalized Banach contractions. Let $\ve>0$, ${\hat{X}}$ be a proper $\ve$-net, $r:X\to {\hat{X}}$ be an $\ve$-projection on ${\hat{X}}$ and $\hat{\mS}:=(\hat{X},(\hat{\phi}_j)_{j=1}^L)$, where $\hat{\phi}_j=(r\circ \phi_j)_{\vert{\hat{X}}}$ is the discretization of $\phi_j$.\\
For any $K\in\K({\hat{X}})$ and $n\in\N$,
\begin{equation}\label{approx2}
h(\overline{F}_{\hat{\mS}}^n(K),A_\mS)\leq\frac{\ve}{1-\alpha_\mS}+\alpha_\mS^nh(K,A_\mS),
\end{equation}
where $A_\mS$ is the attractor of $\mS$.\\
In particular, there is $n_0\in\N$ such that for every $n\geq n_0$, $\overline{F}_{\hat{\mS}}^n(K)$ is an attractor of $\mS$ with resolution $\frac{2\ve}{1-\alpha_\mS}$.
\end{theorem}

\section{Discretization of the Markov Operator for IFSs and GIFSs}\label{sec:discret mark operator IFS}
Throughout this section we assume that $\mS=(X,(\phi_j)_{j=1}^L,(p_j)_{j=1}^L)$ is a GIFSp on a complete metric space $(X,d)$ consisting of $(a_0,...,a_{m-1})$-contractions, for some $a_0,...,a_{m-1}$ which satisfy $\sum_{j=1}^La_j<1$. Let us remark that if $m=1$, i.e., when $\mS$ is an IFSp, then $\mS$ consists of Banach contractions.\\
Consider $\ve>0$, ${\hat{X}}$ a proper $\ve$-net, $r:X\to {\hat{X}}$ a Borel measurable $\ve$-projection on ${\hat{X}}$ and $\hat{\mS}:=(\hat{X},(\hat{\phi}_j)_{j=1}^L, (p_j)_{j=1}^L)$, where $\hat{\phi}_j=(r\circ \phi_j)_{\vert{\hat{X}^m}}$ is the discretization of $\phi_j$.\\
\begin{lemma}
For any bounded set $D\subset X$, the set
  $\{y \in \hat{X} \; | \; r^{-1}(y)\cap D\neq \varnothing\}$
  is finite.
\end{lemma}
\begin{proof}
Choose $y,y'\in \hat{X}$ so that for some $x,x'\in D$, $r(x)=y,\;r(x')=y'$. Then
$$d(y,y')=d(r(x),r(x'))\leq d(r(x),x)+d(x,x')+d(x',r(x'))\leq 2\ve+\on{diam}(D).
$$
Hence the considered set is bounded. Since $\hat{X}$ is proper, we arrive to thesis.
\end{proof}

Consider the family $\Omega \subseteq 2^X$ defined by
  \begin{equation}\label{partition}
    \Omega:=\{ r^{-1}(y) \; | \; y \in \hat{X}\}.
  \end{equation}
This family is obviously a measurable partition of $X$. Moreover, for any $y \in \hat{X}$ we have $r^{-1}(y) \subseteq \overline{B}_{\ve}(y)$ where $\overline{B}_{\ve}(y):=\{ z \in X \; | \;  d(z,y)\leq \ve\}$. Indeed, $z \in r^{-1}(y)$ implies that $r(z)=y$ and by definition  $d(z,y) \leq  \ve$.\\
Now let $e:\hat{X}\to X$ be identity map. Then the $e^\sharp:\mP(\hat{X}) \to \tilde{\mP}(\hat{X})$  is the operator of natural extension of measures from $\mP(\hat{X})$. Indeed, for any $\mu\in\mP(\hat{X})$ and a Borel set $B\subset X$, we have
\begin{equation}
e^\sharp\mu(B)=\mu(e^{-1}(B))=\mu(B\cap \hat{X}).
\end{equation}
Finally, set $\tilde{\mP}(\hat{X}):=e^\sharp(\mP(\hat{X}))$.
The next lemma can be considered as a counterpart of Lemma \ref{filip4} for the setting of measures:
\begin{lemma}\label{net prob}
In the above frame:
\begin{itemize}
  \item[a)] $\tilde{\mathcal{P}}(\hat{X})=\left\{ \sum_{i=1}^{n}   a_i \delta_{y_{i}} \; | \; n\in\N,\;y_{i} \in \hat{X},\; \sum_{i=1}^{m} a_i =1, \; m < \infty\right\}$, where $\delta_x$ is the Dirac measure on $X$ supported on $x$;
  \item[b)] $\tilde{\mathcal{P}}(\hat{X})$ is a proper $\ve$-net of $\mathcal{P}(X)$;
  \item[c)] $r^{\sharp}$ is an $\ve$-projection of $\mP(X)$ to $\tilde{\mP}(\hat{X})$, when considering $r$ as a map $r:X\to X$;
  \item[d)] $\hat{\overline{M}_\mS}=e^\sharp\circ \overline{M}_{\hat{\mS}}$, where $\hat{\overline{M}_\mS}=(r^\sharp\circ \overline{M}_\mS)_{\vert \tilde{\mP}(\hat{X})}$ is the $r^\sharp$-discretization of $M_\mS$.
\end{itemize}
\end{lemma}
\begin{proof}
  (a) Since $\hat{X}$ is proper, its compact sets are finite, and hence probability measures from $\mP(\hat{X})$ has finite support. This, together with the definition of $\tilde{\mP}(\hat{X})$, give (a).\\
(b) Choose $\mu\in\mP(X)$ and let $K\in\K(X)$ be such that $\mu(X\setminus K)=0$. Then the set
  $$\Omega_{\mu}:=\{ r^{-1}(y)\cap K \; | \; y \in \hat{X}\},$$
 is a measurable partition of $K$.

  Consider the set $\hat{K}:=\{y \in \hat{X} \; | \; r^{-1}(y)\cap K\neq \varnothing\}$ . We know that this set is finite, nominally $\hat{K}=\{y_1, ..., y_n\}$. Then we introduce the measure $\nu \in \tilde{\mathcal{P}}(\hat{X})$ by
  \begin{equation}\label{approx of K}
    \nu:= \sum_{i=1}^{n}   \mu(r^{-1}(y_i)\cap K) \, \delta_{y_{i}}.
  \end{equation}
  We claim that $d_{MK}(\mu,\nu) \leq \ve$. To see that consider $f \in \operatorname{Lip}_{1}(X, \mathbb{R})$ . Then
  $$\left|\int_{X} f(x) \mathrm{d} \mu(x)-\int_{X} f(x) \mathrm{d} \nu(x) \right| =
  \left|\sum_{i=1}^{m}\int_{r^{-1}(y_i)\cap K} f(x) \mathrm{d} \mu(x)- \sum_{i=1}^{m}   \mu(r^{-1}(y_i)\cap K) \, f(y_{i}) \right|=
  $$
   $$= \left|\sum_{i=1}^{m}\int_{r^{-1}(y_i)\cap K}  f(x) - f(y_{i}) \mathrm{d} \mu(x) \right|\leq\sum_{i=1}^{m}\int_{r^{-1}(y_i)\cap K}  \left|f(x) - f(y_{i})\right| \mathrm{d} \mu(x) \leq \ve
  $$
  because $f \in \operatorname{Lip}_{1}(X, \mathbb{R})$ and $r^{-1}(y) \subseteq \overline{B}_{\ve}(y)$.

  (c) We first show that the measure $\nu$ considered in previous point equals $r^\sharp\mu$. For a Borel set $B\subset X$, we have
$$
r^\sharp\mu(B)=\mu(r^{-1}(B))=\mu(r^{-1}(B)\cap K)=\mu(r^{-1}(B\cap\{y_1,...,y_n\})\cap K)=\sum_{i:y_i\in B}\mu(r^{-1}(y_i)\cap K)=\nu(B)
$$
and we are done. Thus it remains to prove that $r^\sharp(\mu)=\mu$ for $\mu\in \tilde{\mP}(X)$. If $\mu=\sum_{i=1}^{n}   a_i \delta_{y_{i}}\in\tilde{\mP}(\hat{X})$, then for every Borel set $B\subset X$,
$$
r^\sharp\mu(B)=\mu(r^{-1}(B))=\sum_{i:y_i\in r^{-1}(B)}a_i=\sum_{i:y_i\in (B)}a_i=\mu(B)
$$
so $r^\sharp\mu=\mu$.

  (d) Choose $\mu\in\tilde{\mP}(X)$ and any Borel set $B\subset X$. Then we have:
$$
r^\sharp\circ \overline{M}_\mS(\mu)(B)=\overline{M}_\mS(\mu)(r^{-1}(B))=
M(\mu,...,\mu)(r^{-1}(B))=\sum_{j=1}^Lp_j(\mu\times...\times \mu)(\phi_j^{-1}(r^{-1}(B)))=$$ $$=\sum_{j=1}^Lp_j(\mu\times...\times \mu)((r\circ \phi_j)^{-1}(B))=\sum_{j=1}^Lp_j(\mu\times...\times \mu)(((r\circ \phi_j)_{\vert \hat{X}^m})^{-1}(B))=$$ $$
=\sum_{j=1}^Lp_j(\mu\times...\times \mu)(((r\circ \phi_j)_{\vert \hat{X}^m})^{-1}(B\cap\hat{X}))=\overline{M}_{\hat{\mS}}(\mu_{\vert\hat{X}})(B\cap\hat{X})=e^\sharp\circ \overline{M}_{\hat{\mS}}(\mu).
$$
\end{proof}

Now we give  corollaries of the above lemma and of Theorem \ref{dfp}. Since the theory of IFSs is more widespread than that of GIFSs, we will give two versions separately - for IFSs and GIFSs.\\

A measure $\nu \in\tilde{\mathcal{P}}(\hat{X})$ will be called a \emph{discrete Hutchinson measure for $\mS$ with resolution $\delta$} if  $d_{MK}(\nu, \mu_{\mathcal{S}})\leq \delta$.

\begin{theorem}\label{discrete hutchinson measure2}
Let $(X,d)$ be a complete metric space and $\mS=(X,(\phi_j)_{j=1}^L,(p_j)_{j=1}^L)$ be a IFSp consisting of Banach contractions. Let $\ve>0$, ${\hat{X}}$ be a proper $\ve$-net, $r:X\to {\hat{X}}$ be a proper measurable $\ve$-projection on ${\hat{X}}$ and $\hat{\mS}:=(\hat{X},(\hat{\phi}_j)_{j=1}^L, (p_j)_{j=1}^L)$, where $\hat{\phi}_j=(r\circ \phi_j)_{\vert{\hat{X}}}$ is the discretization of $\phi_j$.\\
For any $\nu \in{{\mathcal{P}}}(\hat{X})$ and $n\in\N$,
\begin{equation}\label{approxDiscreteHutch}
d_{MK}(e^\sharp(M_{\hat{\mathcal{S}}}^n(\nu)),\mu_\mS)\leq\frac{\ve}{1-\alpha_\mS}+\alpha_\mS^n \, d_{MK}(e^\sharp(\nu),\mu_\mS),
\end{equation}
where $\mu_\mS$ is the  Hutchinson measure of $\mS$.\\
In particular,   there is $n_0\in\N$ such that for every $n\geq n_0$, $e^\sharp(M_{\hat{\mathcal{S}}}^n(\nu))$ is a discrete Hutchinson measure of $\mS$ with resolution $\frac{2\ve}{1-\alpha_\mS}$.
\end{theorem}
\begin{theorem}\label{discrete hutchinson measure}
In the frame of the above theorem, assume that $\mS=(X,(\phi_j)_{j=1}^L,(p_j)_{j=1}^L)$ is a GIFSp consisting of $(a_0,...,a_{m-1})$-contractions, where $\sum_{i=0}^{m-1}a_i<1$.\\
For any $\nu \in{{\mathcal{P}}}(\hat{X})$ and $n\in\N$,
\begin{equation}\label{approxDiscreteHutch1}
d_{MK}(e^\sharp(\overline{M}_{\hat{\mathcal{S}}}^n(\nu)),\mu_\mS)\leq\frac{\ve}{1-\alpha}+\alpha^n \, d_{MK}(e^\sharp(\nu),\mu_\mS),
\end{equation}
where $\mu_\mS$ is the  Hutchinson measure of $\mS$ and $\alpha:=\sum_{i=0}^{m-1}a_i$.\\
In particular,  there is $n_0\in\N$ such that for every $n\geq n_0$, $e^\sharp(\overline{M}_{\hat{\mathcal{S}}}^n(\nu))$ is a discrete Hutchinson measure of $\mS$ with resolution $\frac{2\ve}{1-\alpha_\mS}$.
\end{theorem}
Let us explain the thesis. Starting with an IFSp (or with a GIFSp) $\mS=(X,(\phi_j)_{j=1}^L)$ consisting of generalized Banach contractions (or $(a_0,..,a_{m-1})$-contractions), we switch to $\hat{\mS}:=(\hat{X},(\hat{\phi}_j)_{j=1}^{L})$, which is the IFSp (or GIFSp) consisting of discretizations of maps from $\mS$ to $\ve$-net $\hat{X}$. Then, picking any $v\in\mP(\hat{X})$, it turns out that the sequence of iterates $(M_{\hat{\mS}}^n(v))$ (or the sequence of iterates $(\overline{M}_{\hat{\mS}}^n(v))$ in the GIFS case) gives an approximation to the Hutchinson measure $\mu_\mS$ of $\mS$ with resolution $\delta>\frac{\ve}{1-\alpha_\mS}$ (more precisely, their natural extensions from $\tilde{\mP}(\hat{X})$).

\begin{remark}\emph{
In the formulation of the above results we had to distinguish measures from $\mP(\hat{X})$ with their natural extensions from $\tilde{\mP}(\hat{X})$. However, it is rather clear that we can identify $\mu\in\mP(\hat{X})$ with $e^\sharp(\mu)\in\tilde{\mP}(\hat{X})$, and in the remaining part of the paper we will make this identification. This will not lead to any confusions and will simplify notations.}
\end{remark}

\section{Discrete Deterministic Algorithms for Hutchinson Measures}\label{sec:discret determ algor for hutch}
Within this section we assume that $\mS=(X,(\phi_j)_{j=1}^L,(p_j)_{j=1}^L)$ is a GIFSp on a complete metric space $X=(X,d)$ comprising of $(a_0,...,a_{m-1})$-contractions, and $\alpha_{\mS}:=\sum_{i=0}^{m-1}a_i$. Also, $\hat{X}$ is a proper $\ve$-net of $X$, $r:X\to\hat{X}$ is a Borel measurable $\ve$-projection and $\hat{\mS}:=(\hat{X},(\hat{\phi}_j)_{j=1}^L,(p_j)_{j=1}^L)$, where $\hat{\phi}_j:=(r\circ \phi_j)_{\vert X^m}$. All other symbols have the same meaning as in earlier sections.\\
Once again we point out that we formulate all results for GIFSs case, but if $m=1$, then we get the classical IFS case and then $\mS$ consists of Banach contractions, $\overline{F}_\mS=F_\mS$, the Hutchinson operator and $\overline{M}_\mS=M_\mS$, the Markov operator.
\subsection{A description of discrete Markov operator}

\begin{lemma}\label{filip6}
In the above frame, if $\mu=\sum_{i=1}^nv_i\delta_{y_i}\in\mP(\hat{X})$, then \begin{equation}\label{filip7}\on{supp}(\overline{M}_{\hat{\mS}}(\mu))=\{\phi_j(y_{i_0},...,y_{i_{m-1}}):j=1,...,L,\;i_0,...,i_{m-1}=1,...,n\}\end{equation} and enumerating this set by $\{z_1,...,z_{m'}\}$, we have:
\begin{equation}
\overline{M}_{\hat{\mS}}(v)=\sum_{q=1}^{m'}{v'_q}\delta_{z_q},
\end{equation}
where
\begin{equation}\label{definit new weight}
v'_{q}=\sum_{\hat{\phi}_j(y_{i_0},...,y_{i_{m-1}})=z_q} p_j v_{i_0}\cdot...\cdot v_{i_{m-1}}.
\end{equation}
\end{lemma}
\begin{proof}
For every Borel set $B\subset X$,
$$
\overline{M}_{\hat{\mS}}(v)(B)=\sum_{j=1}^Lp_j(\mu\times...\times\mu)(\phi_{j}^{-1}(B))=
\sum_{j=1}^Lp_j\sum_{(y_{i_0},...,y_{i_{m-1}})\in\phi^{-1}(B)}v_{i_0}...v_{i_{m-1}}=
$$
$$
=\sum_{j=1}^Lp_j\sum_{\phi_j(y_{i_0},...,y_{i_{m-1}})\in B}v_{i_0}...v_{i_{m-1}}=\sum_{z_q\in B}v_q'.
$$

\end{proof}
\begin{remark}\label{filip8}\emph{
Looking at the above lemma we see that, in the frames of Theorems \ref{discrete hutchinson measure} and \ref{discrete hutchinson measure2}, when iterating the operator $\overline{M}_{\hat{\mS}}(\mu)$, we automatically get successive iterations of $\overline{F}_{\hat{\mS}}(\on{supp}(\mu))$. In turn, the support of obtained discrete Hutchinson measure $\mu_\delta$ for $\mS$ with resolution $\delta$ is the attractor $A_\delta$ of $\mS$ with resolution $\delta$. This is a discrete version of the classical result where the actual Hutchinson measure of a IFS has support on its attractor.}
\end{remark}
  The next example shows that the discrete Markov operator may not be contractive even if the underlying IFS (of GIFS) consists of (generalized) Banach contractions.
\begin{example}\label{exxx}
Let $X=[0,2]$, and consider it with the euclidean metric. Set $\hat{X}=\left\{\frac{1}{2},1\right\}$ and define
\begin{equation*}
r(x)=\left\{\begin{array}{ccc}1&\mbox{if}&x\neq\frac{1}{2}\\
\frac{1}{2}&\mbox{if}&x=\frac{1}{2}
\end{array}\right..
\end{equation*}
Clearly, $r$ is $1$-projection of $X$ to $\hat{X}$. Now let $\mS=(X,(\phi),(1))$ be the IFSp where
$
\phi(x):=\frac{x}{2}.
$
As can be easily calculated,
\begin{equation*}
\hat{\phi}(x):=(r\circ \phi)_{\vert\hat{X}}(x)=\left\{\begin{array}{ccc}1&\mbox{if}&x=\frac{1}{2}\\
\frac{1}{2}&\mbox{if}&x=1\end{array}\right..
\end{equation*}
Then,
\begin{equation*}
M_{\hat{\mS}}(a\delta_{\frac{1}{2}}+(1-a)\delta_1)=(1-a)\delta_{\frac{1}{2}}+a\delta_1
\end{equation*}
and hence $M^{(2)}_{\hat{\mS}}(\mu)=\mu$, for every $\mu\in\mP(\hat{X})$. In particular, $M_{\hat{\mS}}$ cannot be Banach contraction.
\end{example}

For practical purposes, in presented algorithms we compute the set $F_{\hat{\mS}}(\supp(\nu))$ according to (\ref{filip7}). Then, we enumerate $F_{\hat{\mS}}(\supp(\nu)):=\{z_{q},  1\leq q \leq m'\}$ and, for each $q$ we define coefficients $a_q'$ according to formula (\ref{definit new weight}).

Thus, the output of our algorithm is a bitmap image with the equal shape of $A_{\delta}$ but each pixel represents the measure $\nu_{\delta}$ of the atom $\{y\}$, that is, a gray scale histogram. More than that, the value $\nu_{\delta}(\{y\})$ represents an approximation of the value
\begin{equation}
   \nu_{\delta}(\{y\}) \simeq \mu_{\mS}(r^{-1}(\{y\})).
\end{equation}

\subsection{Uniform $\ve$-nets}

In order to build an algorithm we are going to fix some notation and consider a special type of $\ve$-nets on rectangles in Euclidean spaces. In particular, we assume that $$X=[a_1,b_1]\times...\times[a_d,b_d]\subset \R^d.$$

Given $\varepsilon>0$ we consider the sequence $x[i] \in \mathbb{R}, i \in \mathbb{Z}$ such that $x[i]<x[i+1]$ and $x[i+1]-x[i]=\varepsilon$, for all $i \in \mathbb{Z}$. Then the set $\hat{X} \subseteq \mathbb{R}^{d}$ given by
\begin{equation}\label{net}
  \hat{X}:= \{(x[i_1],..., x[i_d])\in \hat{X} \; | \; i_1,...,i_d \in \mathbb{Z} \}
\end{equation}
is a proper $\ve$-net for $X$ with respect to the Euclidean distance.

We need also to define an  $\ve$-projection on ${\hat{X}}$.
Consider the auxiliary function $q: \mathbb{R} \to \{x[i] \in \mathbb{R}, i \in \mathbb{Z}\}$ given by
\begin{equation}\label{aux_q}
  q_{1}(t):=
  \left\{
    \begin{array}{ll}
       x[i_1], &  \text{ if }   t \leq x[i_1]\\
      x[i], & \text{ if } x[i]\leq t <\frac{x[i]+x[i+1]}{2},\; i\geq m \\
      x[i+1], &  \text{ if }\frac{x[i]+x[i+1]}{2}\leq t \leq x[i+1],\; i+1 \leq {i_{n_1}}\\
     x[i_{n_1}], &  \text{ if }   t \geq x[i_{n_1}]\\
    \end{array}
  \right.
\end{equation}
where $x[i_1]$ is the  smallest point such that $a_{1} \leq x[i]$ and $x[i_{n_1}]$ is the  biggest point such that $b_1 \geq x[i]$. The functions $q_{2},...,q_d$ refer to other coordinates and are defined by the same formula replacing $[a_1,b_1]$ by $[a_i,b_i]$, $i=2,...,q$.

Then, the actual $\ve$-projection on ${\hat{X}}$ is given by
\begin{equation}\label{projection}
  r(x_1,...,x_d):=(q_1(x_1),...,q_d(x_d)).
\end{equation}
We notice that $r$ is clearly proper and Borel measurable because $r^{-1}(v_1,...,v_d)$  is an $F_\sigma$ set  for any $(v_1,...,v_d) \in {\hat{X}}$.
Finally, for a measure $\nu\in\mP(\hat{X})$, let us adopt the notation
\begin{equation}
  \nu=\sum_{k=1}^{n}   \nu_{i^1_{k},...,i^d_{k}} \delta_{(x_1[i^1_{k}],...,x_d[i^d_{k}])},
\end{equation}
where $(x_1[i^1_{k}],...,x_d[i^d_{k}]) \in \hat{X}$ and $\sum_{k=1}^{n}   \nu_{i^1_{k},...,i^d_{k}}=1$.

\subsection{The IFS and GIFS algorithms}

Here we present an algorithm to generate discrete Hutchinson measure for $\mS$ with a desired resolution $\delta$, as well an attractor of $\mS$ with resolution $\delta$.
{\tt
\begin{tabbing}
aaa\=aaa\=aaa\=aaa\=aaa\=aaa\=aaa\=aaa\= \kill
     \> \texttt{GIFSMeasureDraw($\mS$)}\\
     \> {\bf input}: \\
     \> \>  \> $\delta>0$, the resolution.\\
     \> \>  \> $K \subseteq \hat{X}$, any finite and not empty subset (a list of points in $\hat{X}$).\\
     \> \>  \> $\nu$, any probability measure such that $\supp (\nu)= K$. \\
     \> \>  \> The diameter $D$ of a ball in $(X,d)$ containing $A_{\mS}$ and $K$.\\
     \> {\bf output:} \\
     \> \> \>  A bitmap representing a discrete attractor with resolution  at most $\delta$.\\
     \> \> \> A bitmap  image representing a discrete Hutchinson measure
     with resolution  at most $\delta$.\\
     \> {\bf Compute:}\\
     \> \> \> $\displaystyle \alpha_\mS:=\sum_{i=0}^{m-1}a_i$\\
     \> \> \> $\varepsilon >0$ and $N \in \mathbb{N}$ such that $\frac{\ve}{1-\alpha_\mS}+\alpha_\mS^N \, D < \delta$\\
     \> {\bf Initialize}  $\mu:=0$ and $W:=\varnothing$ \\
     \> {\bf for n from 1 to N do}\\
     \>  \>{\bf for $\ell$ from 1 to $\operatorname{Card}(K)$ do}\\
     \>  \>  \>  \>{\bf for $j$ from $1$ to $L$ do}\\
     \>  \>  \>  \>  \>$y_{i_0}:=K[\ell][1]$,..., $y_{i_{m-1}}:=K[\ell][m]$\\
     \>  \>  \>  \>  \>$(x_1[i^1],..., x_d[i^d]):=\hat{\phi}_{j}(y_{i_0},...,y_{i_{m-1}})$\\
     \>  \>  \>  \>  \>If $(x_1[i^1],..., x_d[i^d]) \not\in W$ then $W:=W \cup (x_1[i_1],..., x_d[i_d])$\\
     \>  \>  \>  \>  \>$\nu_{(x_1[i^1],..., x_d[i^d])}:=\nu_{(x_1[i^1],..., x_d[i^d])} +  p_{j} \nu_{y_0}\cdot...\cdot\nu_{y_{m-1}}$\\
     \>  \>  \>  \>{\bf end do} \\
     \>  \>{\bf end do} \\
     \> \> $K:=W$ and $W:=\varnothing$ \\
     \> \> $\nu:=\mu$ and $\mu:=0$\\
     \>{\bf end do}\\
     \>{\bf return: Print} $K$ and $\nu$ \\
\end{tabbing}}

\begin{remark} \label{analogy Hutchinson Measure GIFS}\emph{
By construction, the measure $\mu_{N}:=\overline{M}_{\hat{\mathcal{S}}}^{N}(\nu)$ has support on the finite set $K_{N}:=\overline{F}_{\hat{\mS}}^{N}(K)$ where $K=\supp(\nu)$. Recalling Remark \ref{filip8}, this shows that the discrete Hutchinson  measure $\nu_{\delta}$, with resolution $\delta$, is actually a measure of probability with finite support on a discrete fractal $A_{\delta}$ with resolution $\delta$. }
\end{remark}

\subsection{IFS examples}

  We have run our experiments using \textsc{Fortran 95} implementations of our algorithms. The computer used has an Intel i5-6400T 2.20GHz CPU with 8 GB of RAM.

\begin{example}\label{ex_barns_table}
  Using our algorithms we can recover the results of \cite[Chapter IX]{Bar}, for IFS with probabilities. In page 331, Table IX.1, that author considers the IFSp given by
  \[\mS:
\left\{
  \begin{array}{ll}
      \phi_1(x,y)   & =(0.5 x , \; 0.5 y ) \\
      \phi_2(x,y)   & =(0.5 x + 0.5, \; 0.5 y ) \\
      \phi_3(x,y)   & =(0.5 x , \; 0.5 y + 0.5) \\
      \phi_4(x,y)   & =(0.5 x + 0.5, \; 0.5 y + 0.5) \\
  \end{array}
\right.
\]
with the probabilities $(p_{1}=0.1, \; p_{2}=0.2, \;p_{3}= 0.3, \;p_{4}=0.4)$.  The comparison is made in Figure~\ref{Barns_table_Pic}.

\begin{figure}[ht]
  \centering
  \includegraphics[width=4cm,frame]{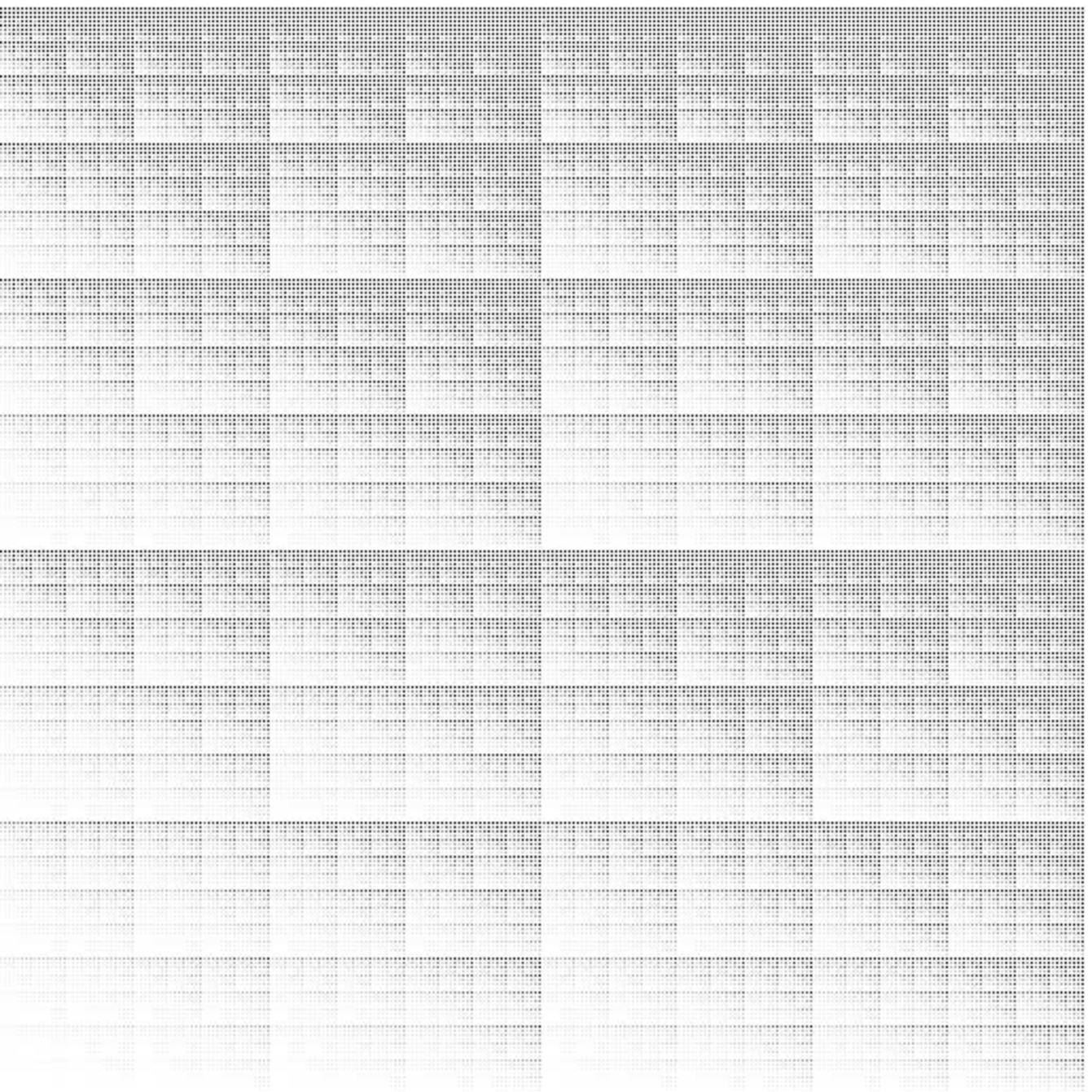}\;
  \includegraphics[width=4cm, height=4cm,frame]{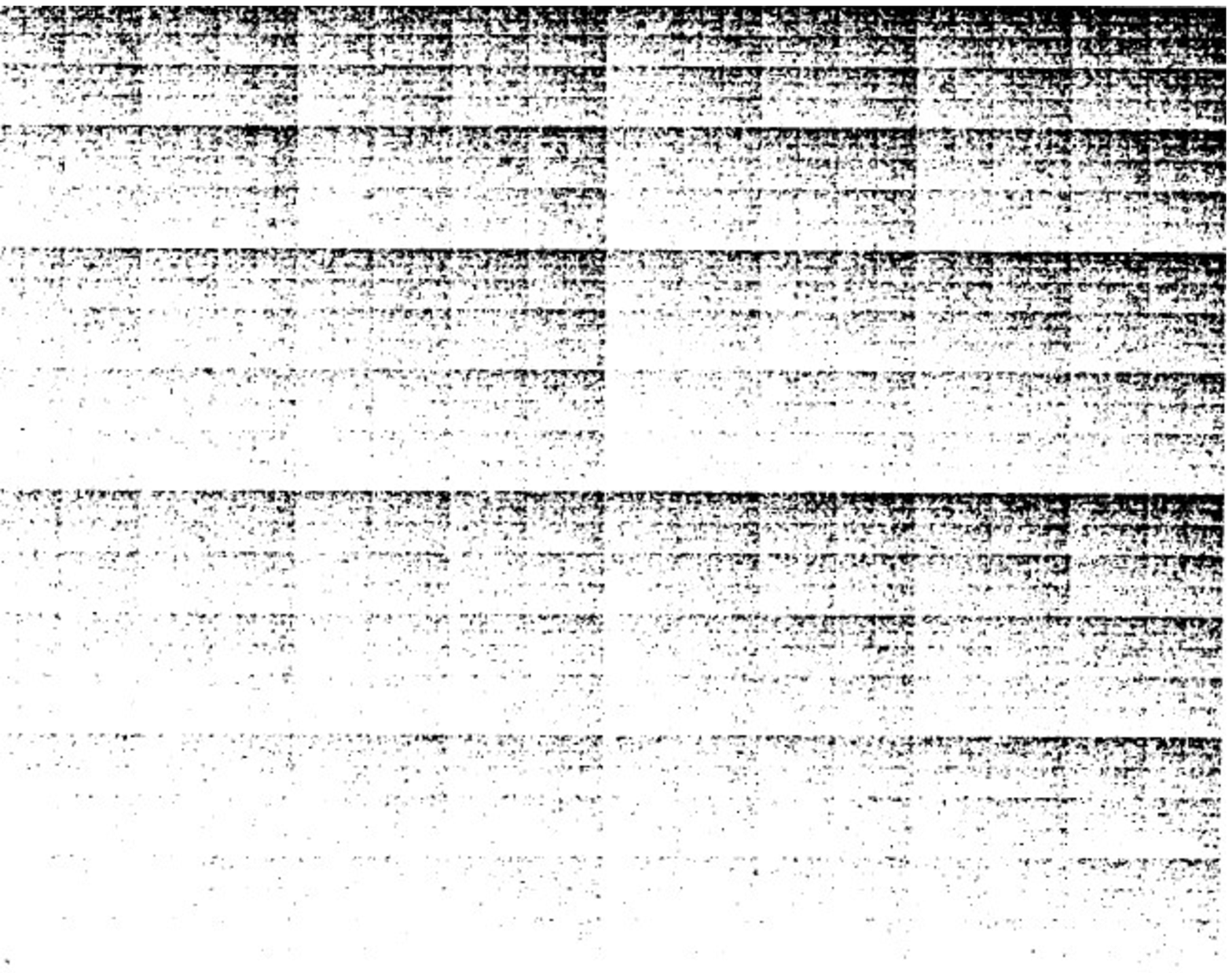}
  \caption{From the left to the right the output of the algorithm \texttt{GIFSMeasureDraw}($\mS$) after 8 iterations, with $512\times512$ pixels and the picture IX.247 obtained in  \cite{Bar} through a random process with 100.000 iterations.}\label{Barns_table_Pic}
  \end{figure}

\end{example}

\begin{example}\label{maple measures} This example is a classic geometric fractal, the Maple Leaf. The approximation of the attractor by the algorithm \texttt{IFSDraw($\mS$)} from \cite{DOS} is presented in Figure~\ref{MapleMeasurePic} and the approximation of the discrete Hutchinson  measure, through \texttt{GIFSMeasureDraw($\mS$)}, is presented in the same figure.  Consider $X=[-2,2]^2$ and the IFS $\phi_1,..., \phi_4: X \to X$  with probabilities $(p_j)_{j=1}^{L=4}$ where
\[\mS:
\left\{
  \begin{array}{ll}
      \phi_1(x,y)   & = (0.8 x + 0.1, 0.8 y + 0.04)\\
      \phi_2(x,y)   & = (0.5 x + 0.25, 0.5 y + 0.4)\\
      \phi_3(x,y)   & = (0.355 x - 0.355 y +0.266,  0.355 x + 0.355 y + 0.078)\\
      \phi_4(x,y)   & = (0.355 x + 0.355 y +0.378,  -0.355 x + 0.355 y + 0.434)\\
  \end{array}
\right.
\]
  As we can see, when the initial probabilities are small on the index of a map responsible for a part of the fractal attractor, the Hutchinson measure is very little concentrated on that part. For example, if we choose equal probability  we will have a much more equal distribution as in Figure~\ref{MapleMeasurePicUnif}.
 \begin{figure}[ht]
  \centering
  \includegraphics[width=4cm,frame]{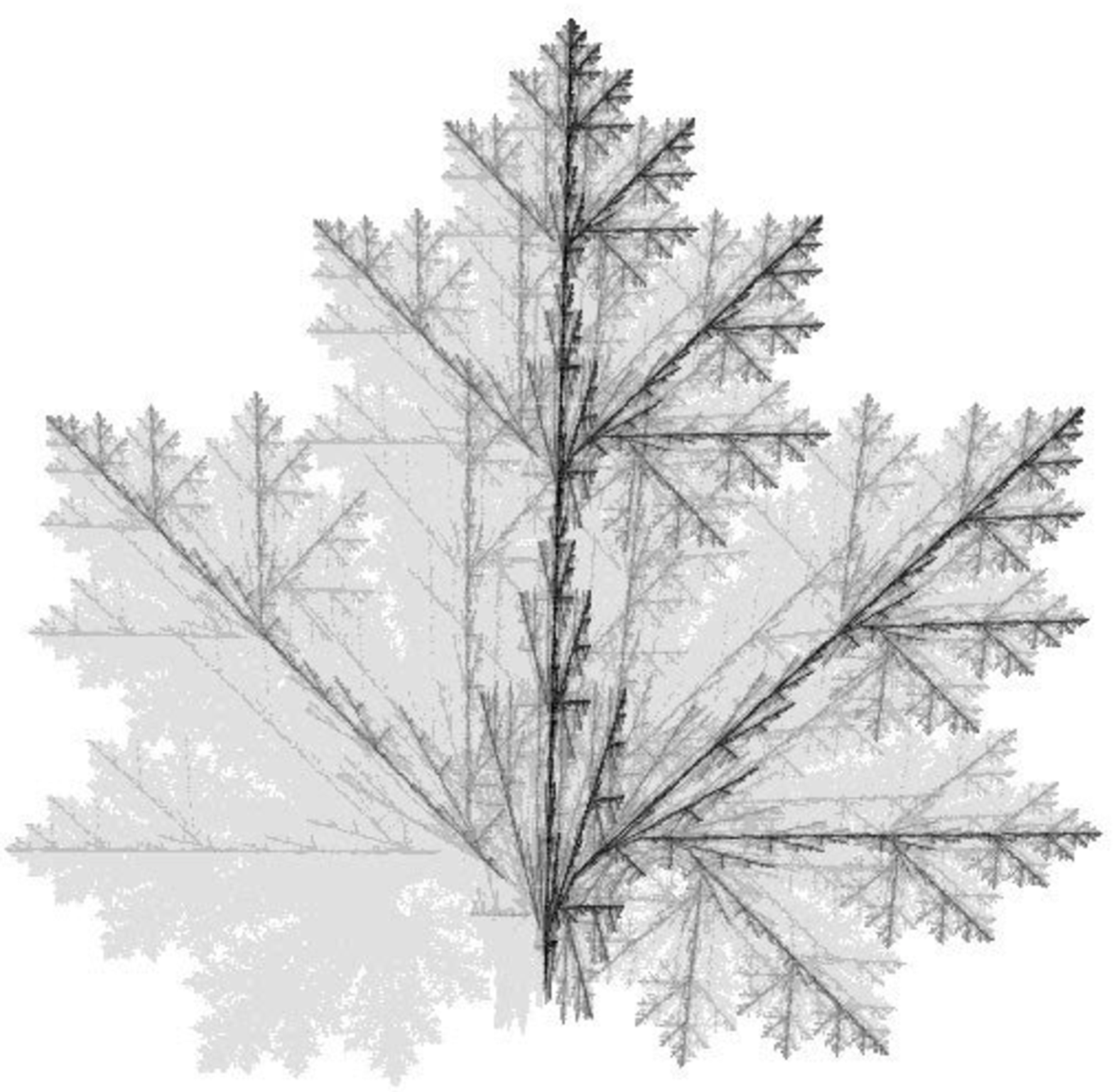}\;
  \includegraphics[width=4cm,frame]{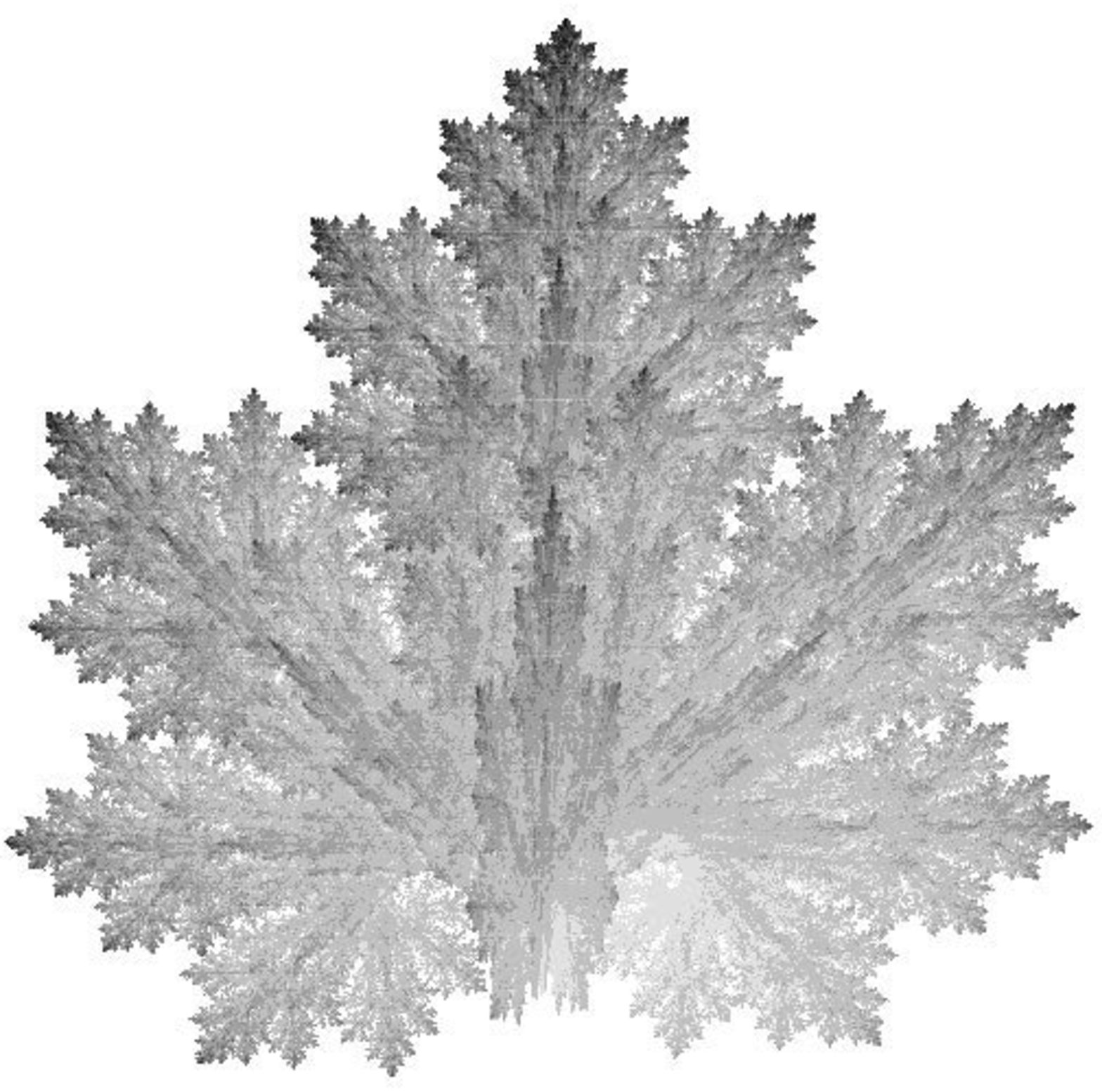}\;
  \includegraphics[width=4cm,frame]{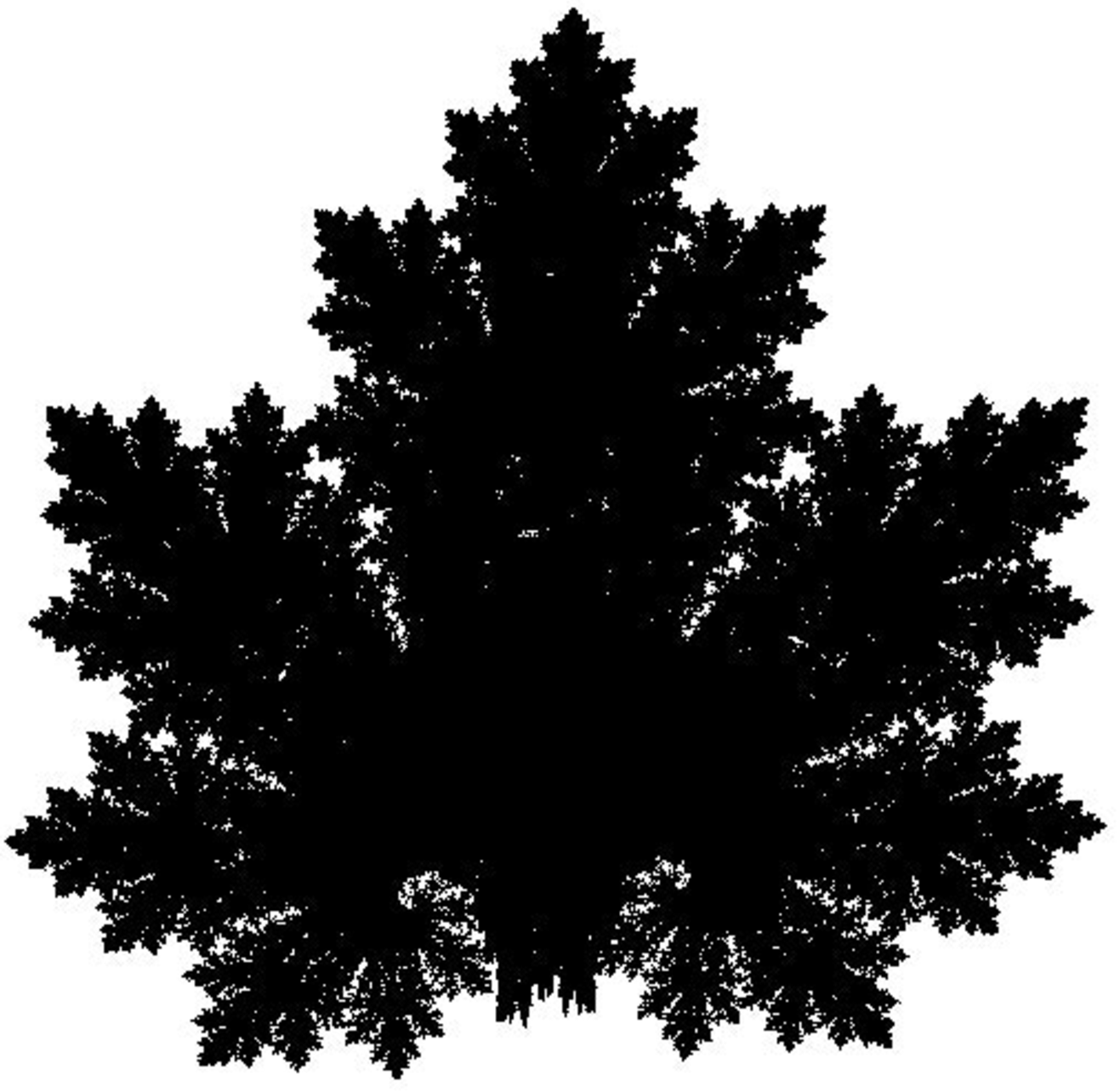}
  \caption{From left to right, the output of the algorithm \texttt{GIFSMeasureDraw}($\mS$) after 5 iterations, $512\times512$ pixels, with the set of probabilities $(p_{1}=0.3, \; p_{2}=0.2, \;p_{3}= 0.05, \;p_{4}=0.45)$ and $(p_{1}=0.05, \; p_{2}=0.2, \;p_{3}=0.3, \;p_{4}=0.45)$ respectively, for the rightmost, algorithm \texttt{IFSDraw($\mS$)} after 12 iterations.}\label{MapleMeasurePic}
  \end{figure}

\begin{figure}[ht]
  \centering
  \includegraphics[width=4cm,frame]{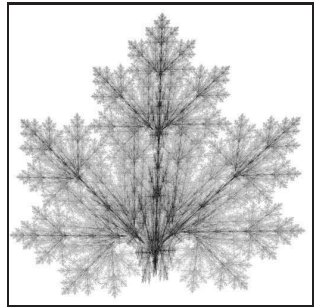}
  \caption{The output of the algorithm \texttt{GIFSMeasureDraw}($\mS$) after 12 iterations, $\ve=0.004$ (512 points in a uniform net) with the set of probabilities $(p_{1}=0.25, \; p_{2}=0.25, \;p_{3}= 0.25, \;p_{4}=0.25)$. In this case the measure $\nu$ is a Hutchinson measure with resolution inferior to $\delta=0.2441701363$.}\label{MapleMeasurePicUnif}
  \end{figure}
\end{example}

 \begin{example}\label{ex:galatolo_final} This example is presented in the final section of \cite{GMN}. Each map consists of a translation followed by rotation and a dilation by a factor smaller than one, thus ensuring that it is a Banach contraction and hence our algorithm is applicable. The approximation of the discrete Hutchinson measure, through \texttt{GIFSMeasureDraw($\mS$)}, is presented in the Figure~\ref{Galatolo_final_measure}. An approximation of the attractor, as a secondary output of the algorithm, is presented in Figure~\ref{Galatolo_final_ex}. Consider $X=[0,1]^2$ and the IFS $\phi_1,..., \phi_4: X \to X$  with probabilities $(p_j)_{j=1}^{L=4}$ where
\[\mS:
\left\{
  \begin{array}{ll}
      \phi_1(x,y)   & = (\cos(\frac{\pi}{6}) (0.4 x -0.24)-\sin(\frac{\pi}{6}) (0.4 y -0.08)+0.6, \\
      & \;\cos(\frac{\pi}{6}) (0.4 y -0.08)+\sin(\frac{\pi}{6}) (0.4 x -0.24)+0.2)\\
      \phi_2(x,y)   & = (\cos(\frac{\pi}{30}) (0.6 x -0.03)+\sin(\frac{\pi}{30}) (0.6 y -0.12)+0.05, \\
      & \; \cos(\frac{\pi}{30}) (0.6 y -0.12)-\sin(\frac{\pi}{30}) (0.6 x -0.03)+0.2)\\
      \phi_3(x,y)   & = (0.5 x -0.475+0.95, \;0.5 y -0.475+0.95)\\
      \phi_4(x,y)   & =(0.45 x -0.045+0.1, \; 0.45 y -0.405+0.9)\\
  \end{array}
\right.
\]
and $(p_{1}=0.18, \; p_{2}=0.22, \;p_{3}= 0.3, \;p_{4}=3)$.\\
As the maps are affine, a simple evaluation shows that $\alpha_\mS=0.6$, where $\mS=(X,(\phi_j)_{j=1}^4,(p_j)_{j=1}^4)$ ($\alpha_\mS$ is the maximum contraction between the four maps $\phi_j$, all being uniform contractions). After 15 iterations of our algorithm, which took only 5.56 seconds for $1024\times 1024$ pixels, we obtain (Figure~\ref{Galatolo_final_measure}, top-right) the associated Hutchinson (invariant) measure with resolution  at most $\delta$, where $\delta \sim \frac{\ve}{1-\alpha_\mS}+\alpha_\mS^N \, D =0.00310635 $, for $N=15$, $\ve=\frac{1}{1024}$ and $D=\sqrt{2}$. Following the technique of \cite{GMN} a similar approximation (Figure~\ref{Galatolo_final_measure}, bottom) is given with resolution of at most $0.0047583$, to compare. Most specifically, we have the following benchmark data given in the below table:\\
\begin{center}
\begin{tabular}{|c|c|c|c|}
  \hline
  Pixels $M\times M$ & Iter. Number $N$ & Time(in seconds)& Resol. at most $\delta \sim \frac{\ve}{1-\alpha_\mS}+\alpha_\mS^N \, D$  \\
  \hline
  $256 \times 256 $ & 15 & 0.56 &   0.010430566982011563\\
  $512 \times 512 $ & 15 & 1.76 &   0.005547754482011564\\
  $1024 \times 1024 $ & 15 & 5.56 & 0.003106348232011563\\
  \hline
\end{tabular}
\end{center}
 \begin{figure}[ht]
  \centering
  \includegraphics[width=4cm,,height=4cm,frame]{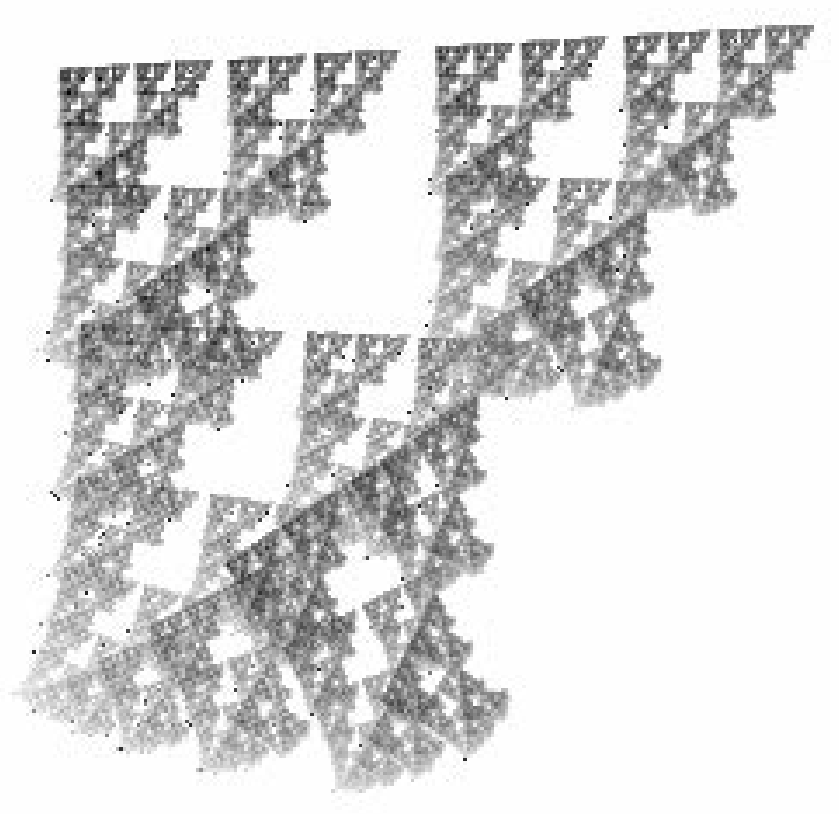}\;
  \includegraphics[width=4cm,,height=4cm,frame]{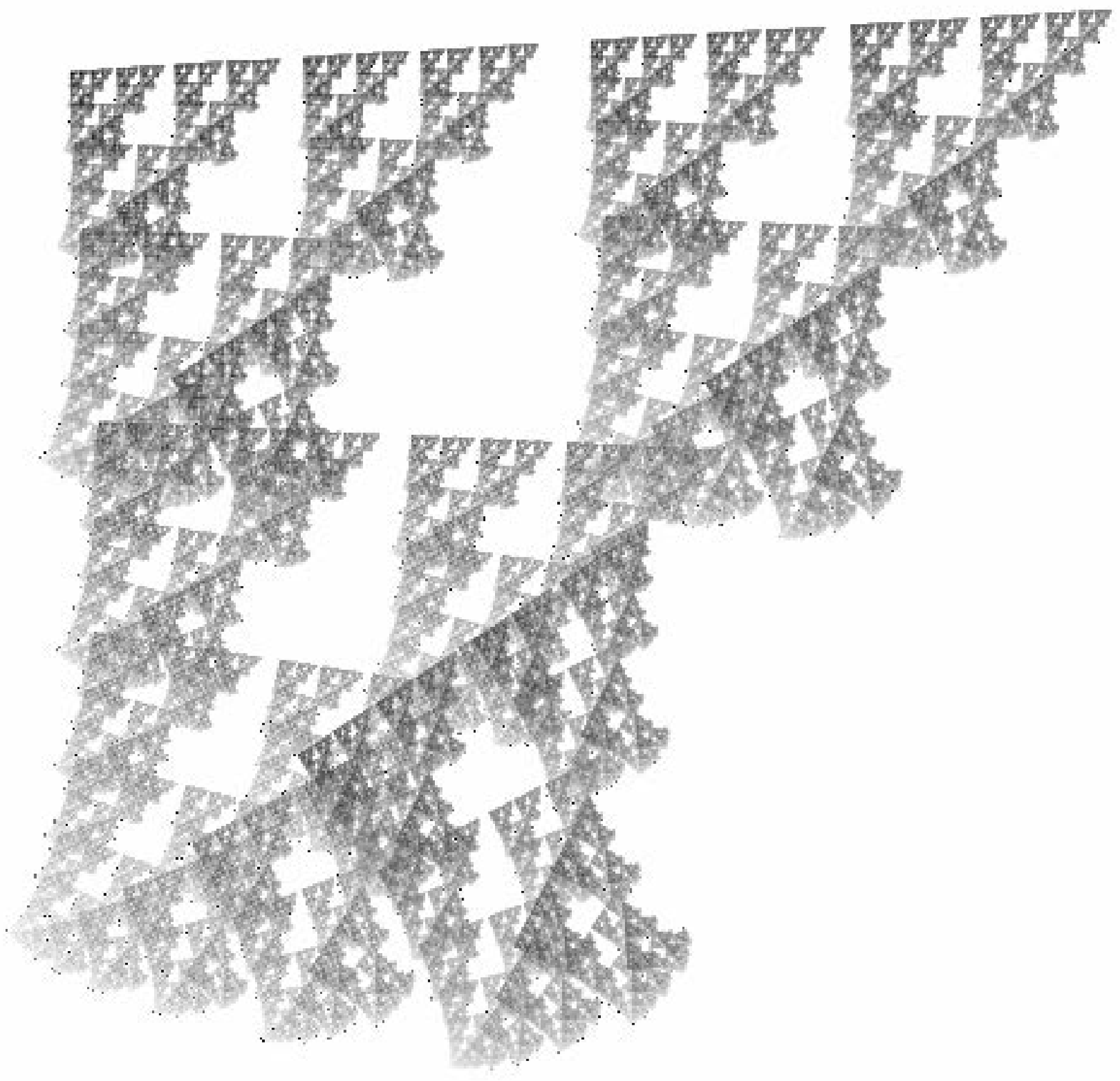}\;
  \includegraphics[width=4cm,,height=4cm,frame]{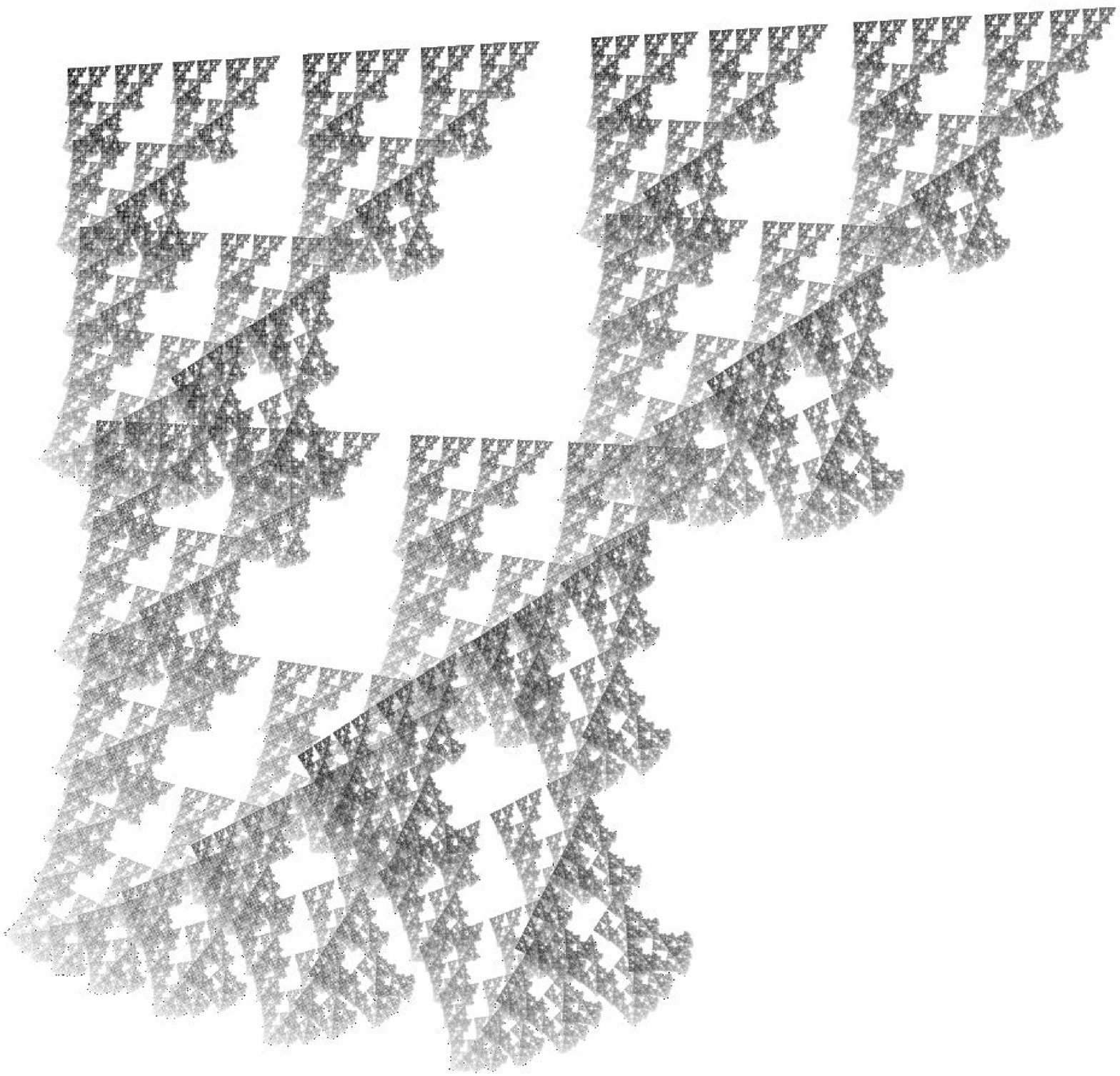}\\
  \includegraphics[width=4cm,,height=4cm,frame]{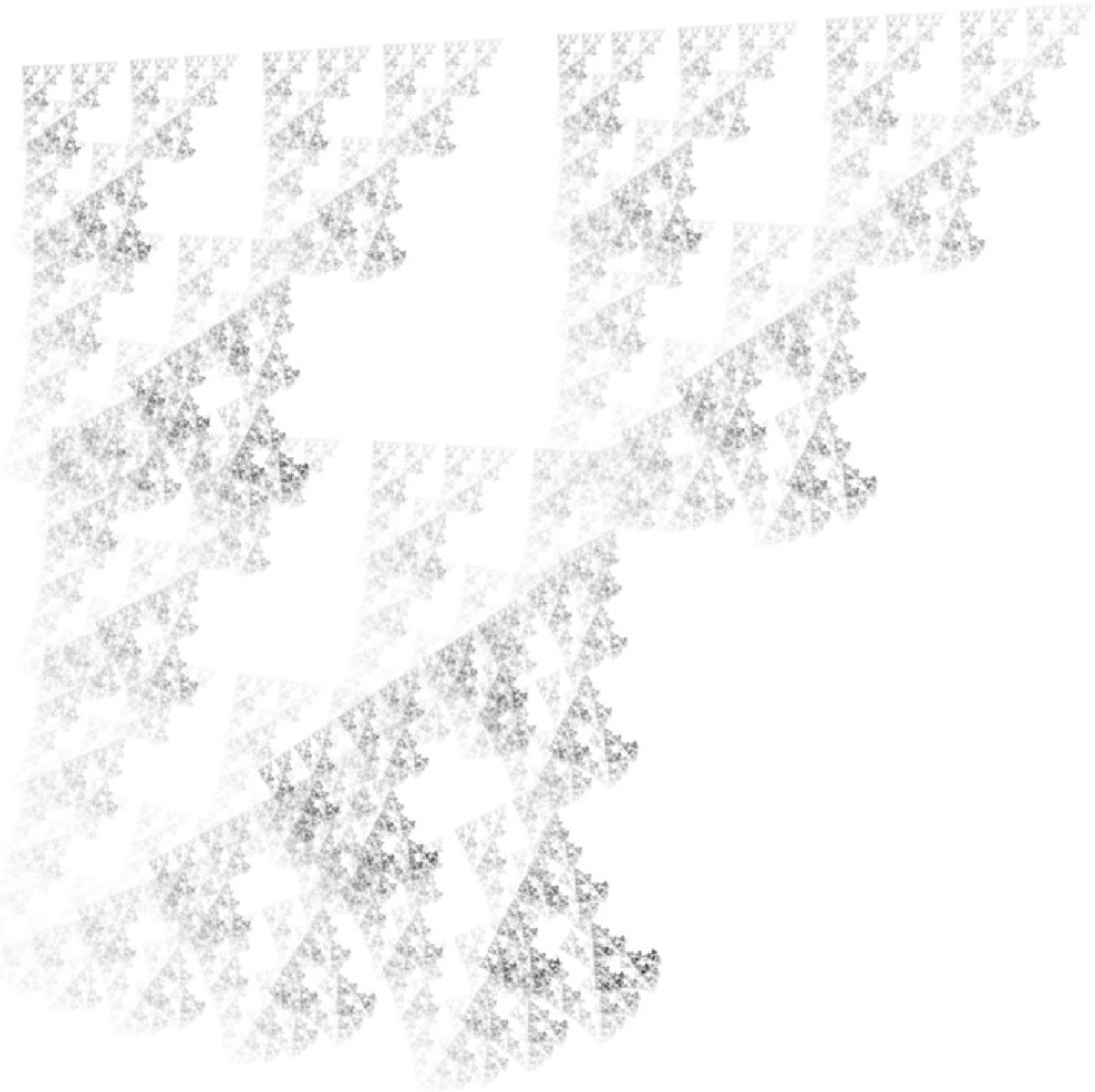}
  \caption{ From left to right, the output of the algorithm \texttt{GIFSMeasureDraw}($\mS$) after 15 iterations, for $256\times 256$, $512\times 512$ and $1024\times 1024$ pixels, with the prescribed set of probabilities; below, the picture given in \cite{GMN}.}\label{Galatolo_final_measure}
  \end{figure}
\begin{figure}[ht]
  \centering
  \includegraphics[width=4cm,frame]{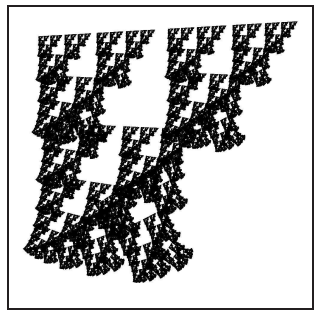}
  \caption{ The attractor $A_\mS$, as the secondary output of the algorithm \texttt{GIFSMeasureDraw}($\mS$), for the same configurations as in  Figure~\ref{Galatolo_final_measure}. }\label{Galatolo_final_ex}
  \end{figure}
\end{example}

\subsection{GIFS examples}
\begin{example}\label{example filip measures} The approximation of the attractor by the algorithm \texttt{GIFSDraw($\mS$)} from \cite{DOS} is presented in Figure~\ref{FilipMeasurePic} and the approximation of the discrete Hutchinson  measure, through \texttt{GIFSMeasureDraw($\mS$)}, is presented in the same figure.  Consider $X=[0,1]^2$ and the GIFS $\phi_1,..., \phi_3: X^2 \to X$  with probabilities $(p_j)_{j=1}^{L=3}$ where
\[\mS:
\left\{
  \begin{array}{ll}
       \phi_1(x_1,y_1,x_2,y_2)&=(0.25 x_1+0.2 y_2, 0.25 y_1+0.2 y_2) \\
       \phi_2(x_1,y_1,x_2,y_2)&=(0.25 x_1+0.2 x_2, 0.25 y_1+0.1 y_2+0.5)\\
       \phi_3(x_1,y_1,x_2,y_2)&=(0.25 x_1+0.1 x_2+0.5, 0.25 y_1+0.2 y_2))\\
  \end{array}
\right.
\]
  Once again, we can see that when the initial probabilities are small on the index of a map responsible for a part of the fractal attractor, the Hutchinson measure is very little concentrated on that part.
 \begin{figure}[ht]
  \centering
  \includegraphics[width=4cm,frame]{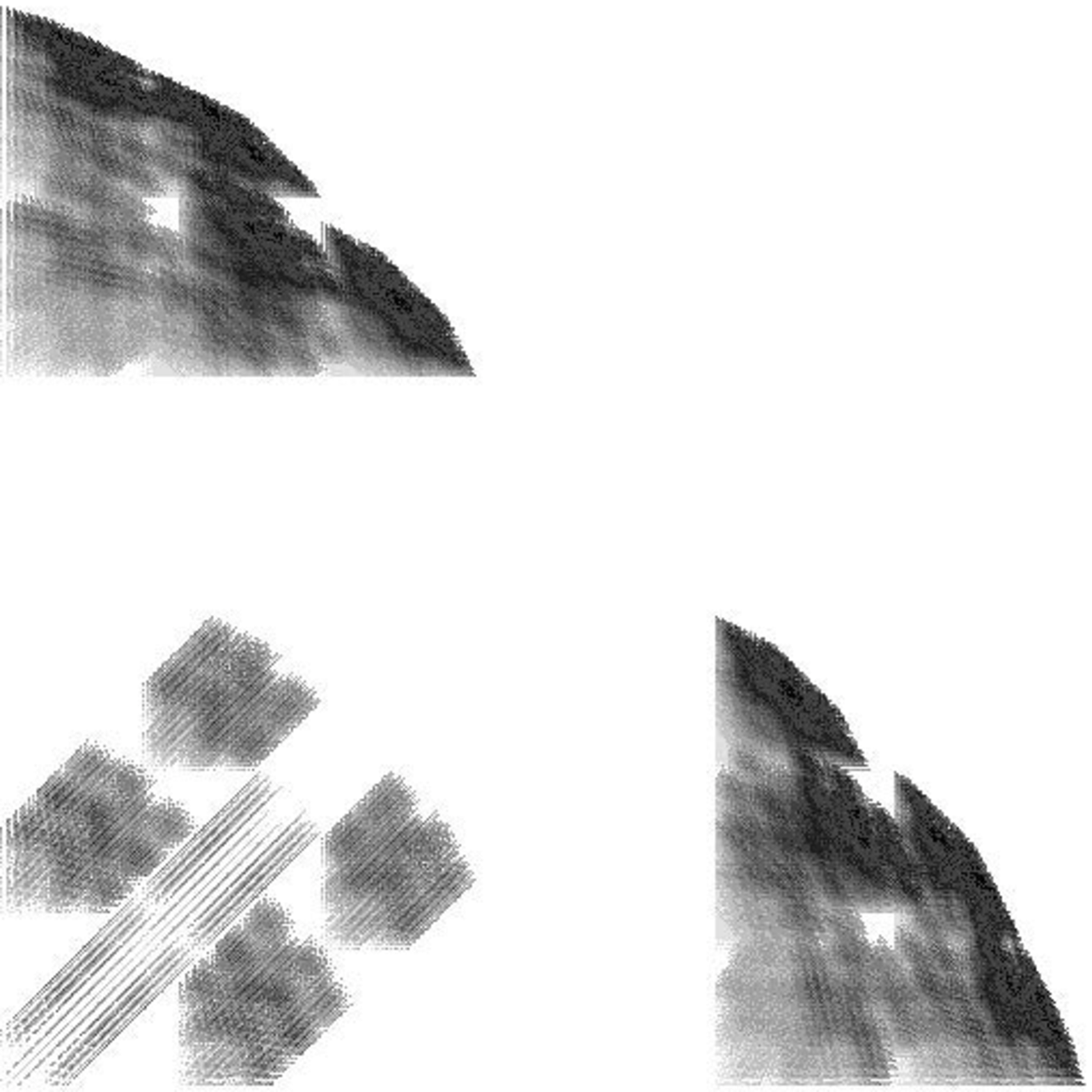}\;
  \includegraphics[width=4cm,frame]{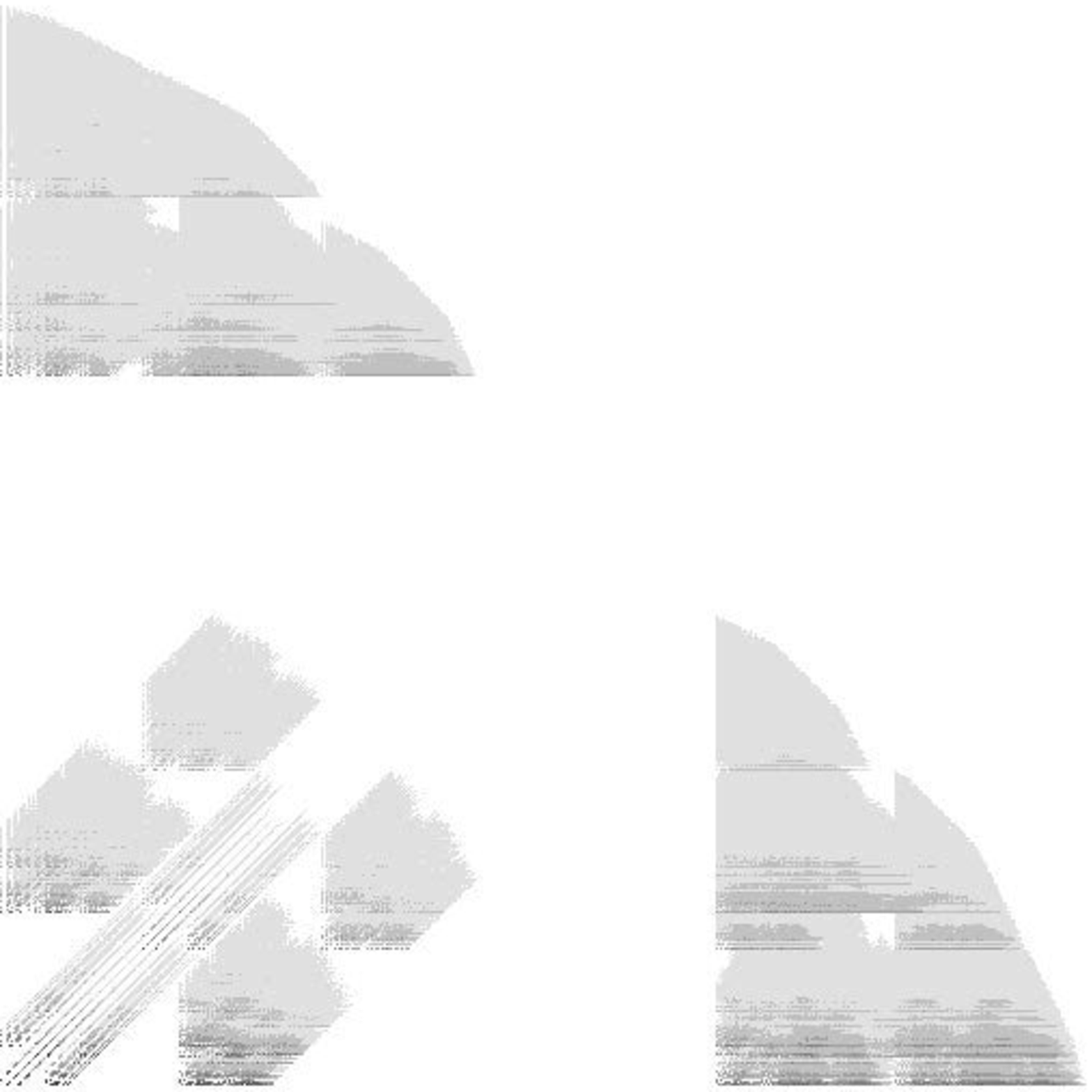}\;
  \includegraphics[width=4cm,frame]{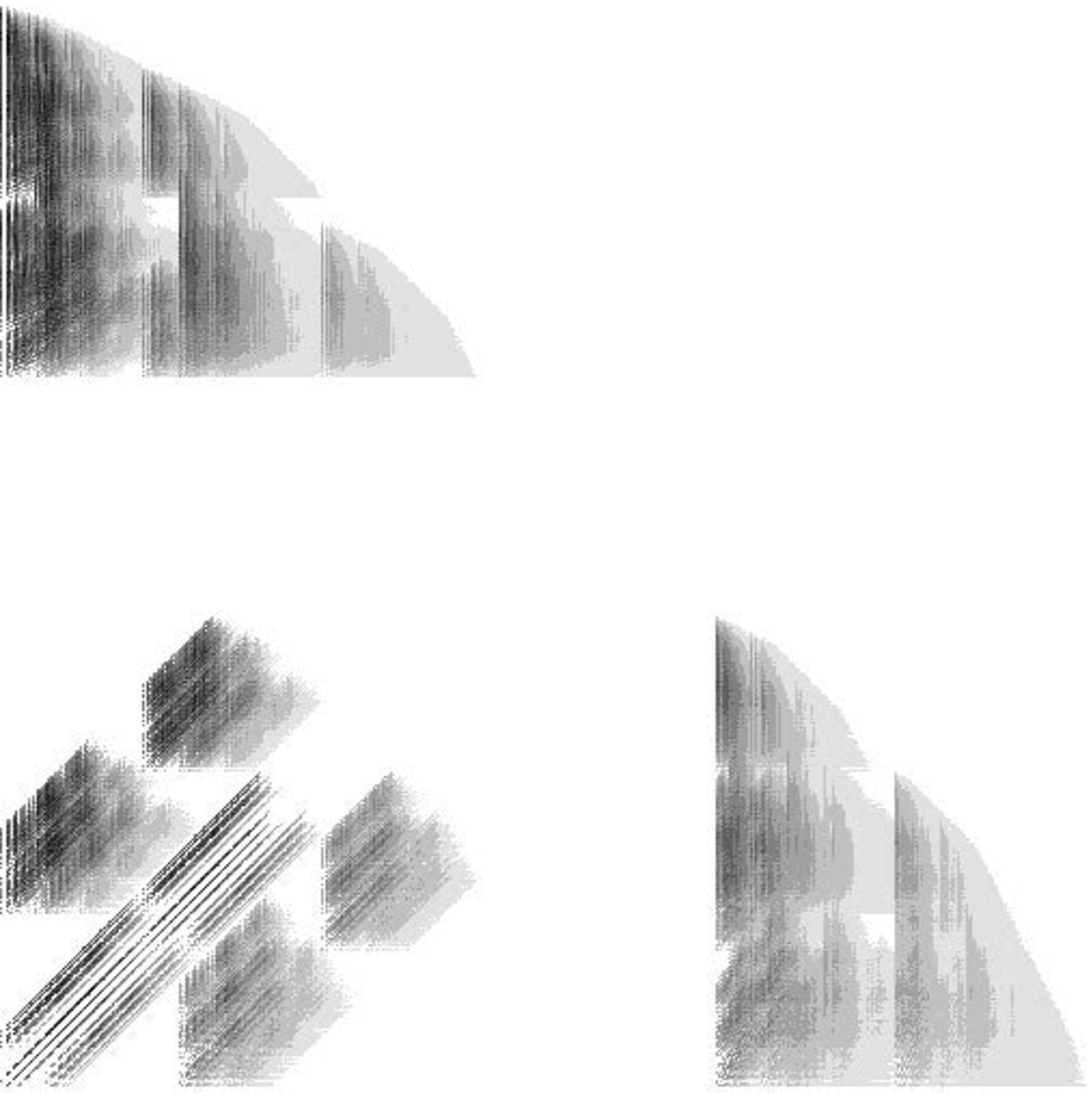}\\
  \includegraphics[width=4cm,frame]{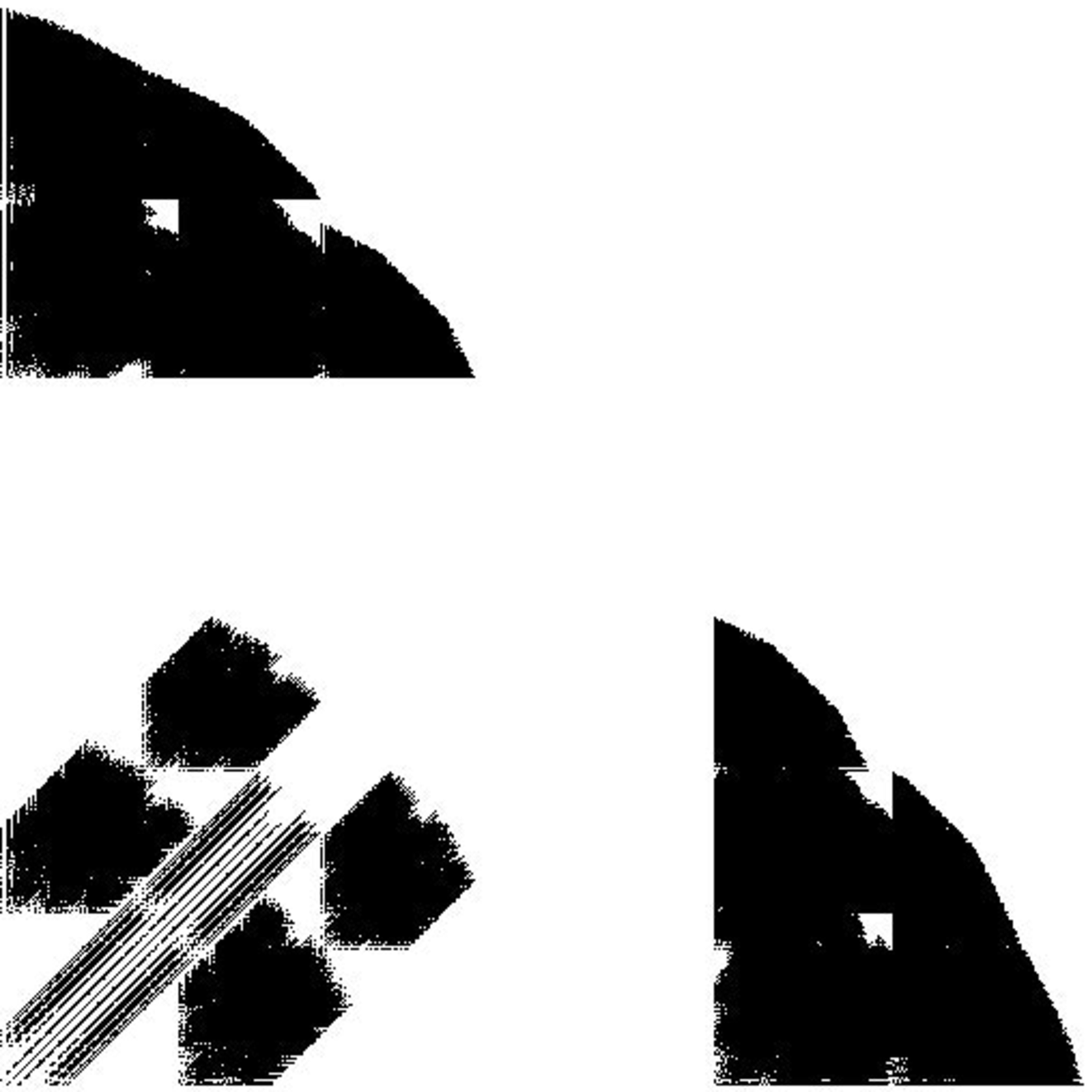}
  \caption{On the top, from the left to the right the output of the algorithm \texttt{GIFSMeasureDraw($\mS$)} after 12 iterations, with the set of probabilities $(p_{1}=0.10, \; p_{2}=0.45, \;p_{3}= 0.45)$, $(p_{1}=0.45, \; p_{2}=0.10, \;p_{3}= 0.45)$ and $(p_{1}=0.45, \; p_{2}=0.45, \;p_{3}= 0.10)$ respectively. On the bottom, algorithm \texttt{GIFSDraw($\mS$)} after 12 iterations. In all the cases we have , $512\times512$ pixels.}\label{FilipMeasurePic}
  \end{figure}
\end{example}

\section{Further applications} \label{sec:Further applications}

\subsection{Approximating integrals with respect to stationary
Probability measures}\label{sec:approx_integrals}

In a recent preprint \cite{CJ} the authors describe a method to approximate integral of functions with respect to stationary probability measures (measures that are fixed points for the Markov operator associated to an IFS with probabilities, $M_{\mathcal{S}}\left(\mu_{\mathcal{S}}\right)=\mu_{\mathcal{S}}$) which are the Hutchinson measures for those IFSs. The setting is the interval $[0,1]$ and the IFS is required to fulfill some additional regularity properties such as holomorphic extension, control of derivatives of the maps in the IFS and on the set of functions that one may integrate.

Since our algorithm \texttt{GIFSMeasureDraw}($\mS$) for dimension $1$ provides a discrete $\delta$-approximation of such measures $\mu_{\mS}$ in the form
$$
  \nu:=\sum_{k=1}^{n}   \nu_{i_{k}} \delta_{x[i_{k}]}
$$
where the points $x[i_{k}]$ are in the correspondent $\delta$-approximation of the actual attractor, we can approximate the integral of a Lipschitz function $g:[0,1] \to \mathbb{R}$ by
$$\int_{[0,1]} g d\mu_{\mS} \simeq \sum_{k=1}^{n}   \nu_{i_{k}} g(x[i_{k}])$$
with precision $\delta$.

As a demonstration of this we next describe the measures and the integrals for three examples found in \cite{CJ} using our algorithms.

\begin{example}\label{ex:Italo:moment}
 In this first example we consider the Hausdorff moment of the Hutchinson measure $\mu_{\mS}$ which is given by
 $$
\gamma_{n}=\gamma_{n}(\mu_{\mS}):=\int_{-\infty}^{\infty} x^{n} \mathrm{d} \mu_{\mS}(x), n=0,1, \ldots
$$
\[\mS:
\left\{
  \begin{array}{ll}
       \phi_1(x)&= \frac{1}{3}  x \\
       \phi_2(x)&= \frac{1}{3}  x+\frac{2}{3}\\
  \end{array}
\right.
\]
with probabilities $p_1:=1/3$ and $p_2:=2/3$.

For purpose of comparison we use the algorithm \texttt{GIFSMeasureDraw}($\mS$) to approximate $\mu_{\mS}$ (see Figure~\ref{measure_moment}) and then compute $\gamma_{n}, n=0,1, \ldots, 10$:
\begin{figure}[ht]
  \centering
  \includegraphics[width=3cm,frame]{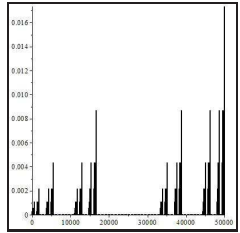}
  \caption{Histogram of $\mu_{\mS}$ produced by algorithm \texttt{GIFSMeasureDraw}($\mS$) with resolution $\delta=0.0001$ after 10 iterations, having a high definition $50,000$ pixels taking 1.9 seconds.}\label{measure_moment}
\end{figure}
The first $10$ moments are displayed in Table~\ref{tab:moment}.

\begin{table}
  \centering
   \begin{tabular}{|c|c|c|}
     \hline
     $n$ & Approx.  $\sum_{k=1}^{m}   \nu_{i_{k}} (x[i_{k}])^n$ & Actual value  $\int_{[0,1]} x^n d\mu_{\mS}$, from \cite{CJ}. \\
     \hline
     0 & 1.0000000 & 1 \\
     1 & 0.6666608 & 2/3=0.66666666666... \\
     2 & 0.5555478 & 5/9=0.55555555555... \\
     3 & 0.4957168 & 58/117=0.49572649572... \\
     4 & 0.4552590 &  799/1755=0.45527065527... \\
     5 & 0.4246685 &  \vdots \\
     6 & 0.4002248 &  \\
     7 & 0.3800868 &  \vdots\\
     8 & 0.3631560 &  1213077397297/3340208879865=0.36317411303...\\
     9 & 0.3486940 & 764170684622650/2191399705783431=0.34871351064...\\
    10 & 0.3361713 & 16313445679660723325/48524163685162512633=0.33619220694...\\
     \hline
   \end{tabular}
  \caption{Moments using the $\delta$-approximation of $\mu_{\mS}$ by \texttt{GIFSMeasureDraw}($\mS$)}\label{tab:moment}
\end{table}

\end{example}

\begin{example}\label{ex:Italo:wass}
In this second example we consider the Wasserstein distance between two  Hutchinson measures $\mu_{\mS_1}$ and $\mu_{\mS_2}$ (see Figure~\ref{measure_wass}) associated to different probabilities for the same IFS
\[\mS:
\left\{
  \begin{array}{ll}
       \phi_1(x)&= \frac{\sin (\pi x / 4)}{6}+\frac{1}{4} \\
       \phi_2(x)&= \frac{\sin (\pi x / 4)}{3}+\frac{2}{3}\\
  \end{array}
\right.
\]
with probabilities $p_1:=1/7, p_2:=6/7$ and $q_1:=1/2, q_2:=1/2$ respectively.

It is easy to see that the IFS is $\frac{\pi}{12}$-Lipschitz. Then our algorithm \texttt{GIFSMeasureDraw}($\mS$) can be used to approximate the integrals and, in particular, under the hypothesis of \cite{CJ} the Wasserstein  distance is given by
$$
W_{1}\left(\mu_{\mS_1}, \mu_{\mS_2}\right)=\left|\int x \mathrm{d} \mu_{\mS_{2}} - \int x \mathrm{d} \mu_{\mS_1}\right| \simeq 0.22104594557263850324...
$$
\begin{figure}[ht]
  \centering
  \includegraphics[width=3cm,frame]{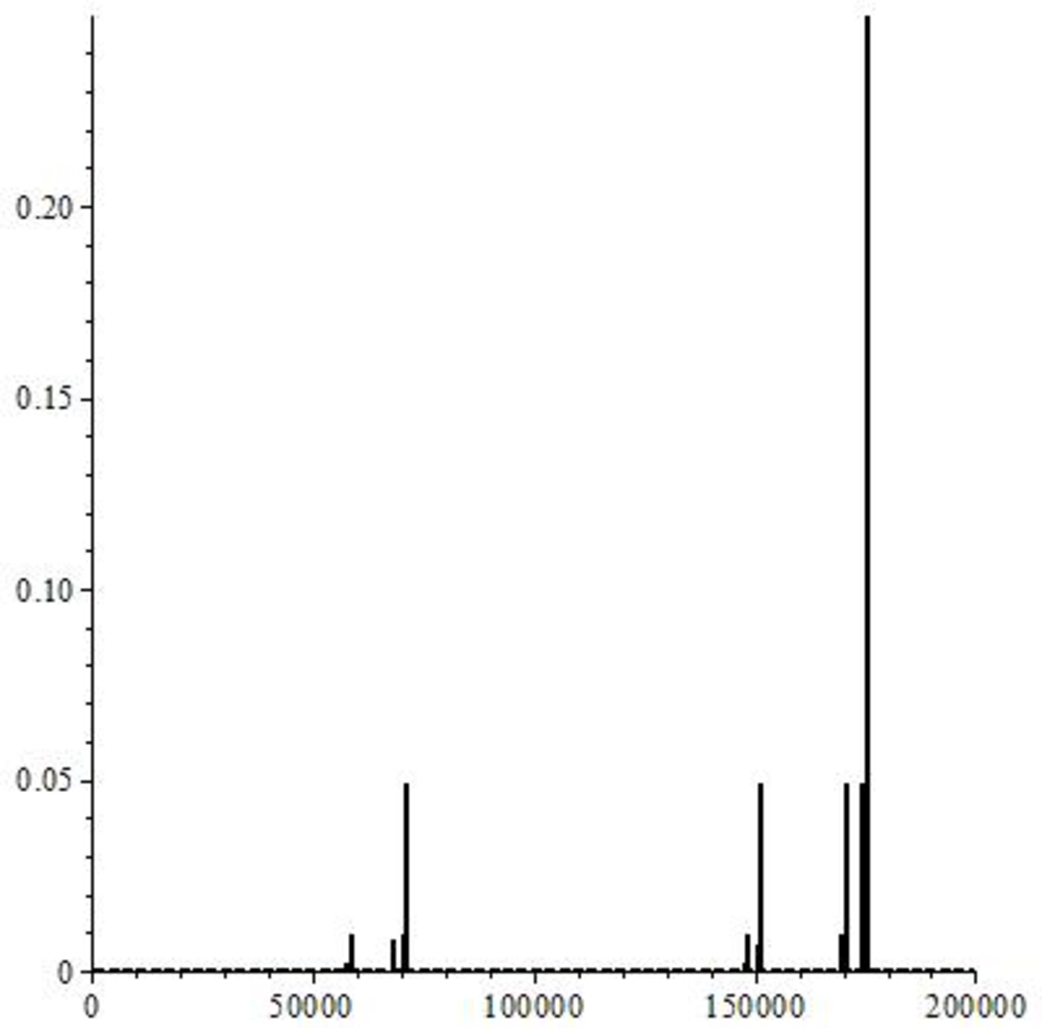}
  \includegraphics[width=3cm,frame]{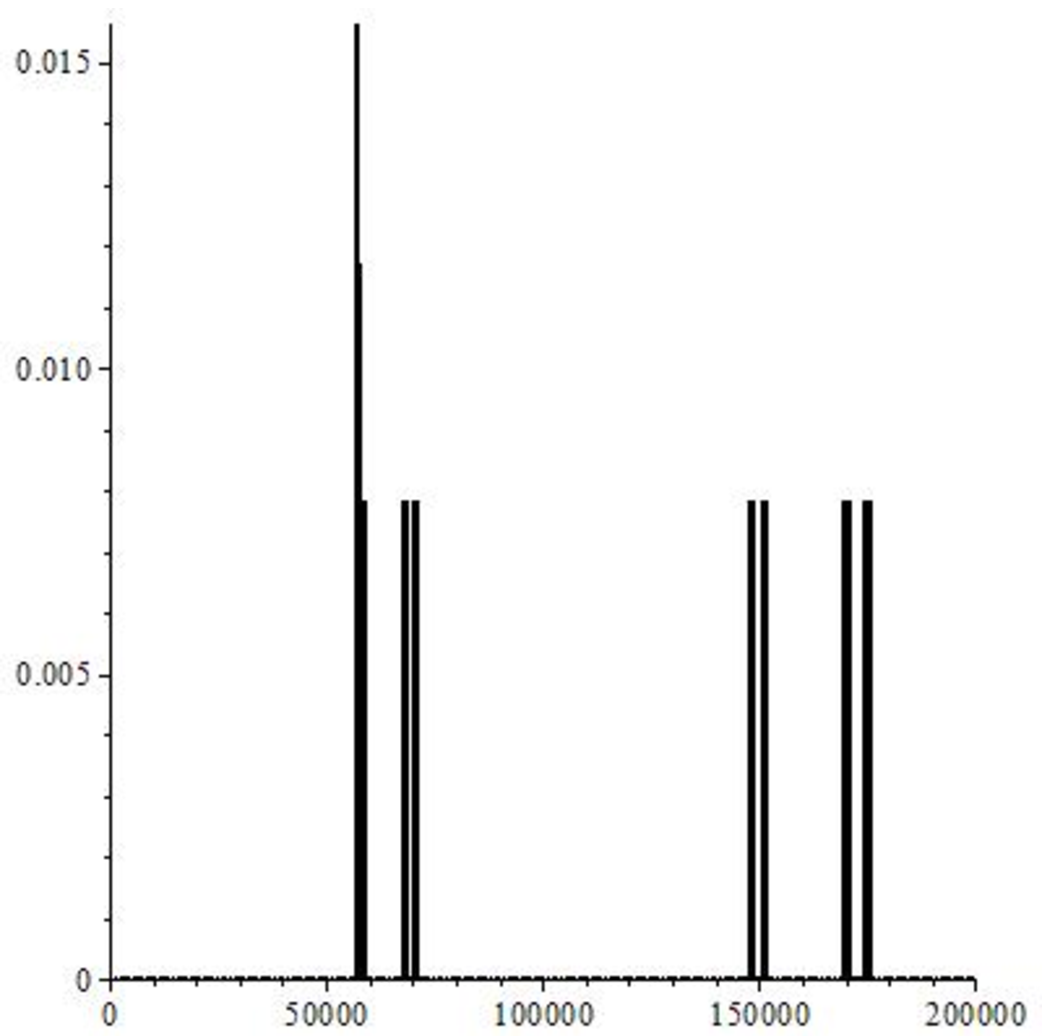}
  \caption{From left to right: the histograms of $\mu_{\mS_1}$ and $\mu_{\mS_2}$ produced by algorithm \texttt{GIFSMeasureDraw}($\mS$) with resolution $\delta=0.00003$ after 10 iterations, having a high definition $200,000$ pixels taking 5.8 seconds.}\label{measure_wass}
\end{figure}

\end{example}

\begin{example}\label{ex:Italo:lyapun}
For the last, we consider the problem of compute the Lyapunov exponent of the Hutchinson measure $\mu_{\mS}$ (see Figure~\ref{measure_lyap}) of a IFSp
given by
$$
\chi_{\mu_{\mS}} :=-\int \sum_{i=1}^{2} p_{i} \log \left|\phi_{i}^{\prime}(x)\right| \mathrm{d} \mu_{\mS}(x) \simeq 1.7367208099326368...
$$
for
\[\mS:
\left\{
  \begin{array}{ll}
       \phi_1(x)&= \frac{\sin (\pi x / 4)}{6}+\frac{1}{4}\\
       \phi_2(x)&= \frac{\sin (\pi x / 4)}{3}+\frac{2}{3}\\
  \end{array}
\right.
\]
with probabilities $p_1:=1/3, p_2:=2/3$.
\begin{figure}[ht]
  \centering
  \includegraphics[width=3cm,height=3cm,frame]{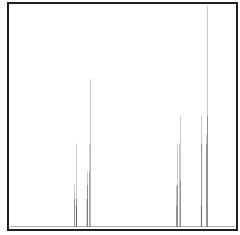}
  \caption{The histogram of $\mu_{\mS}$ produced by algorithm \texttt{GIFSMeasureDraw}($\mS$) with resolution $\delta=2.11\times10^{-6}$ after 20 iterations, having a high definition $3,200,000$ pixels taking 0.5690 seconds.}\label{measure_lyap}
\end{figure}
\end{example}

\subsection{Projected Hutchinson measures}
One can easily adapt algorithm \texttt{GIFSMeasureDraw}($\mS$) to also compute the integral of a function with respect to the Hutchinson measure and compare this result against the typical averages, as predicted by Elton's ergodic theorem, see \cite{Elt}.

For example, given the GIFS $\mS$ from \cite[Example 11]{EO},
\[\mS:
\left\{
  \begin{array}{ll}
       \phi_1(x,y)&= \frac{1}{3}  x+\frac{1}{4}   y \\
       \phi_2(x,y)&= \frac{1}{3}  x-\frac{1}{4}   y+ \frac{1}{2}\\
  \end{array}
\right.
\]
we consider its extension
\[\mR:
\left\{
  \begin{array}{ll}
       \psi_1(x,y)&=(y, \frac{1}{3}  x+\frac{1}{4}   y) \\
       \psi_2(x,y)&=(y, \frac{1}{3}  x-\frac{1}{4}   y+ \frac{1}{2})\\
  \end{array}
\right.
\]
which is an eventually contractive IFS on $[0,1]^2$ (the second power of $\mR$ is $0.5435541904$-Lipschtz), so our theory works. In both cases we consider probabilities $p_1:=0.65$ and $p_2:=0.35$. As pointed out in \cite{EO}, if $\nu \in \mP([0,1])$ is the Hutchinson measure for the GIFSp  $\mS=(X,(\phi_j)_{j=1}^2,(p_j)_{j=1}^2)$ and $\mu_{\mR} \in \mP([0,1]^2)$ is the Hutchinson measure for the IFSp $\mR=(X,(\psi_j)_{j=1}^2,(p_j)_{j=1}^2)$ (see Figure~\ref{ElismAttracIFSPic}) then $\nu \neq \nu'$, where $\nu'$ is the projected Hutchinson measure $\Pi_1^{\sharp} (\mu_{\mR})$.

Using our algorithms we are capable to display a histogram representation of such distributions on $[0,1]$. In each case the height of each vertical bar represents the approximate measure of a cell in the $\ve$-net with $350$ points in $X=[0,1]$, and resolution $2\times 10^{-2}$ (see Figure~\ref{ElismAttracGIFSPic}).
\begin{figure}[ht]
  \centering
  \includegraphics[width=3cm,frame]{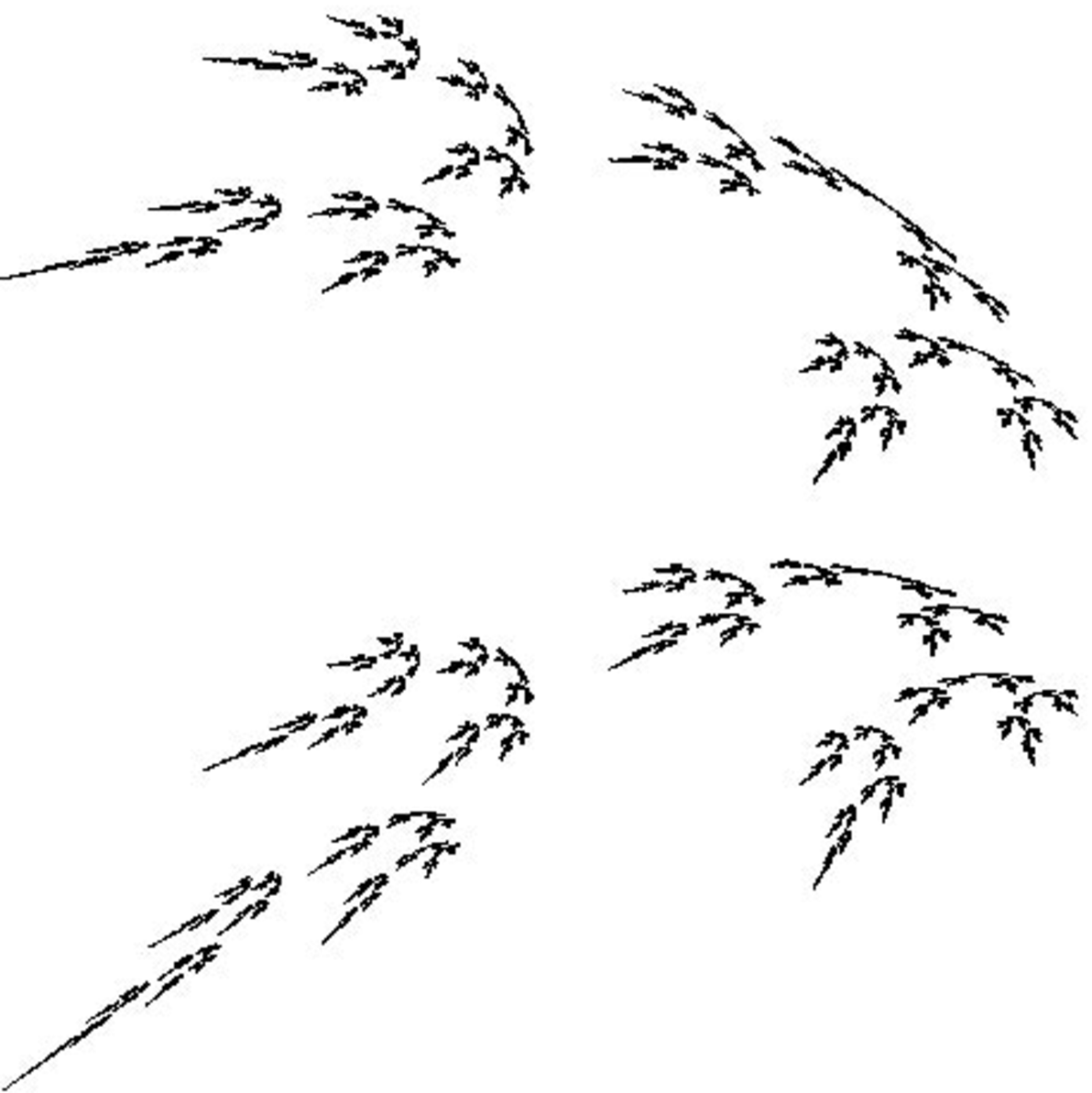}\;
  \includegraphics[width=3cm,frame]{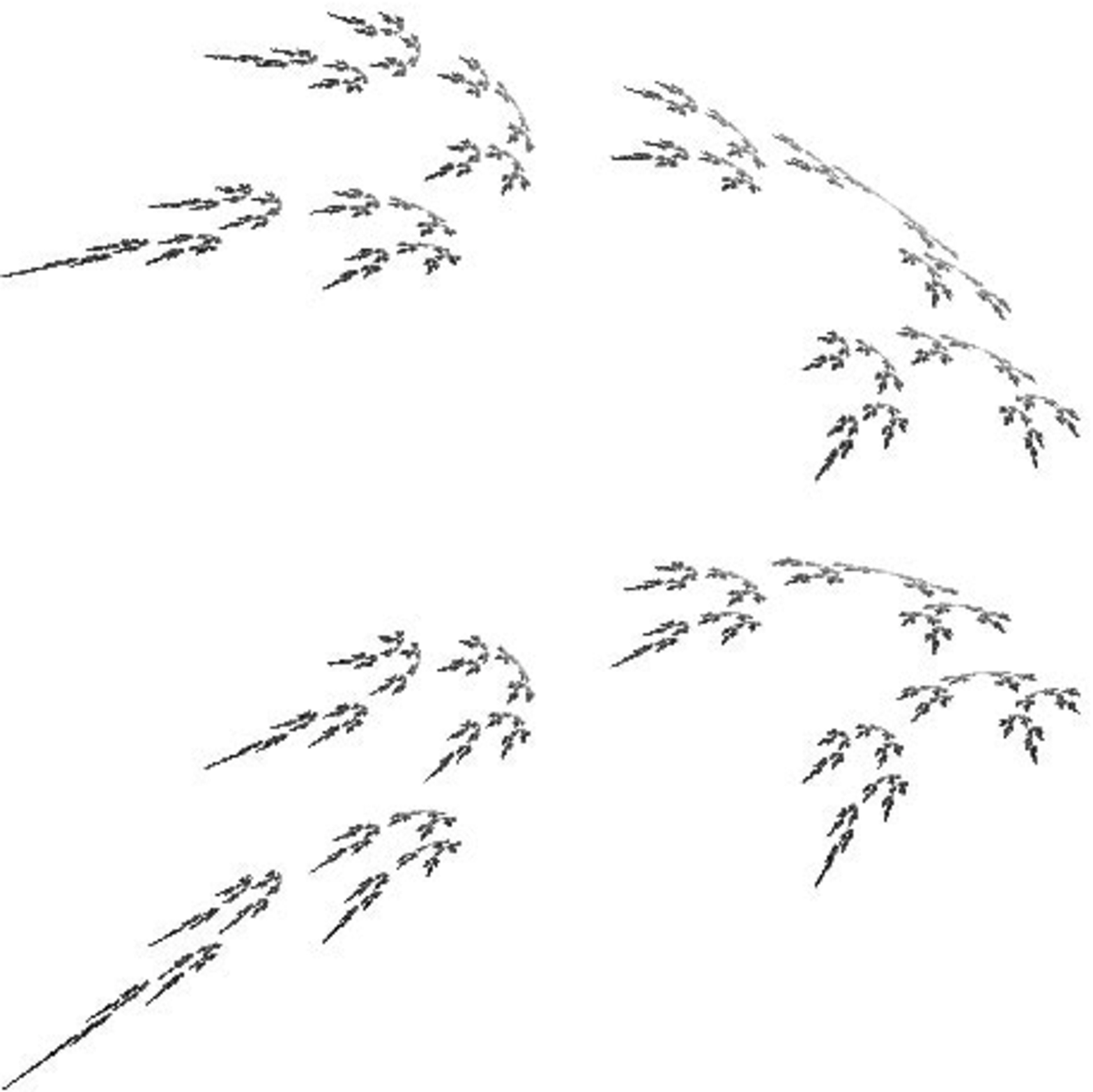}
  \caption{The output of algorithm \texttt{GIFSMeasureDraw}($\mR$) after 25 iterations, having a fairly high definition, $512\times512$ pixels. On the left the attractor $A_{\mR}$ on the right the histogram of $\mu_{\mR}$.}\label{ElismAttracIFSPic}
\end{figure}

\begin{figure}[ht]
  \centering
  \includegraphics[width=3cm,frame]{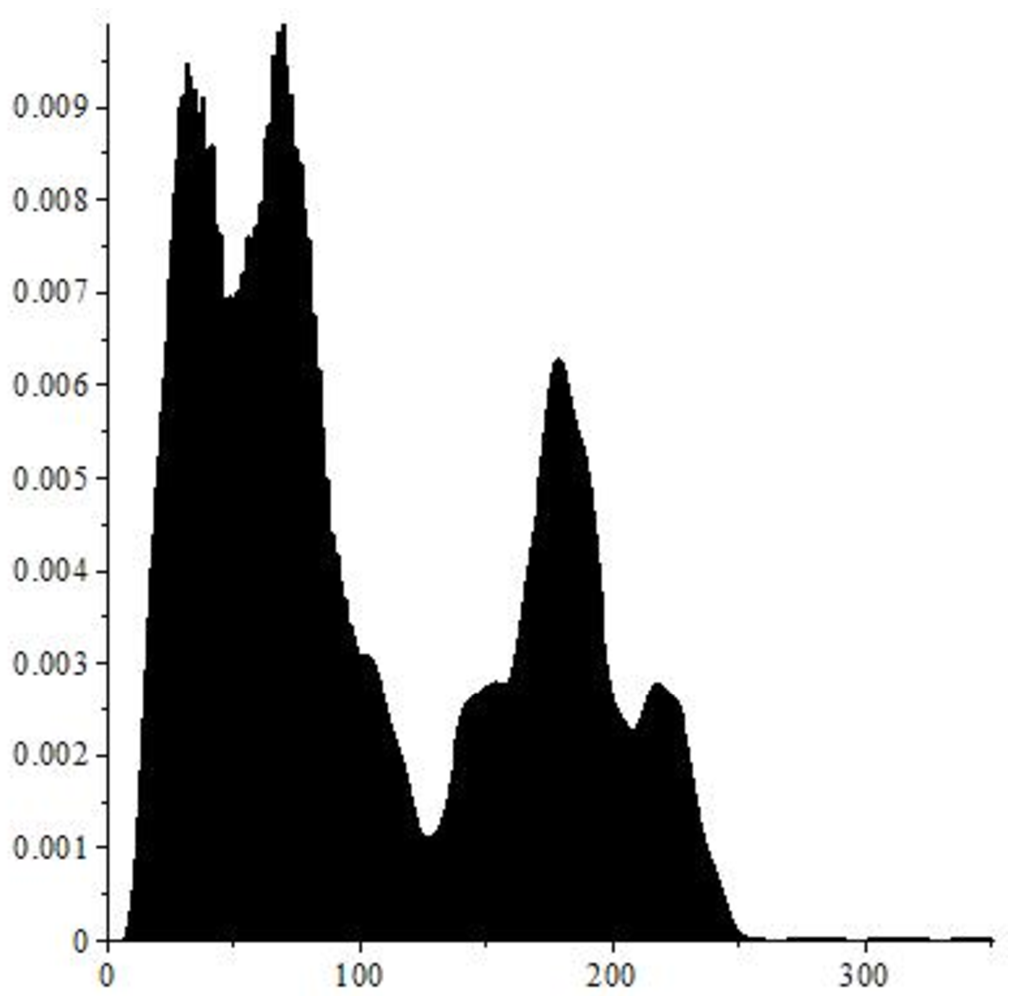}\;
  \includegraphics[width=3cm,frame]{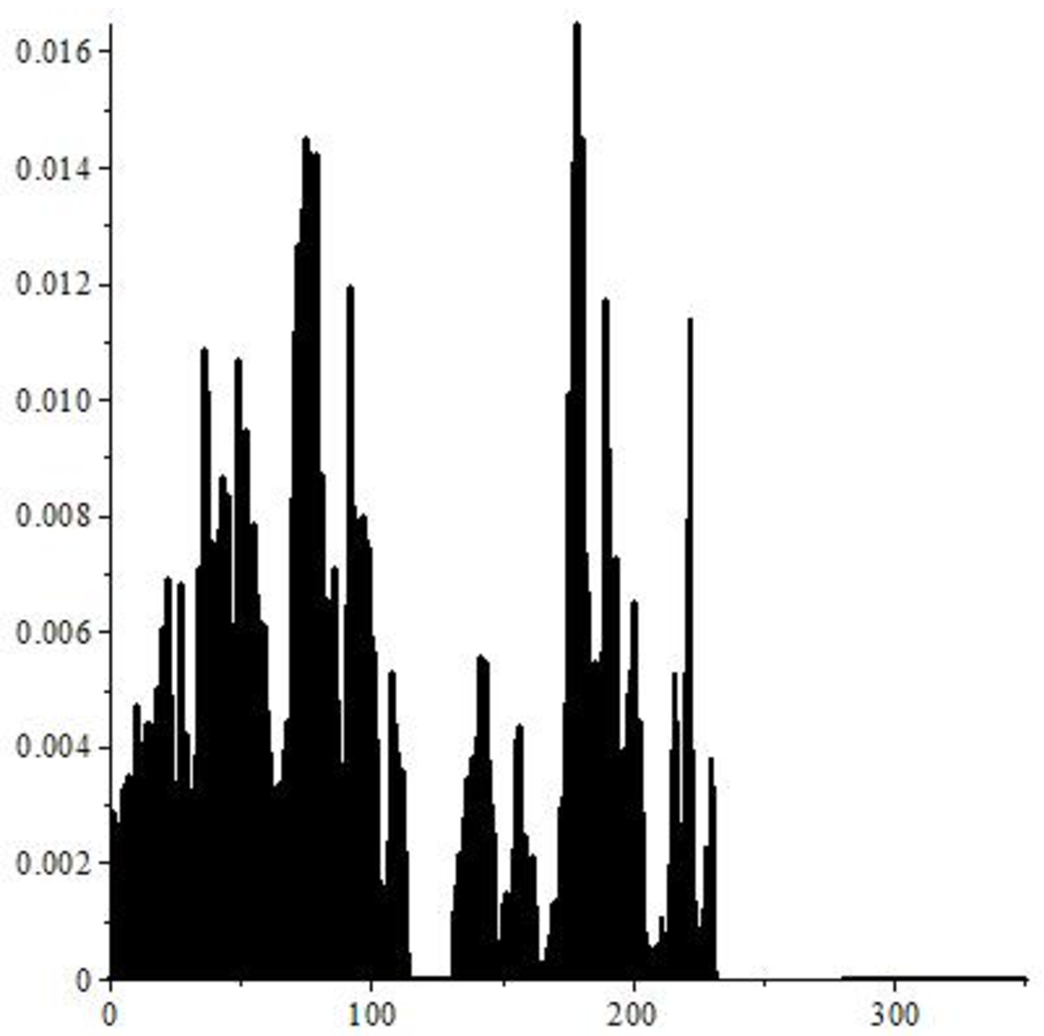}
  \caption{The output of algorithm \texttt{GIFSMeasureDraw}($\mS$) after 25 iterations, having a fairly high definition $350\times 350$ pixels. On the left the attractor $\nu$ on the right the histogram of $\nu'=\Pi_1^{\sharp} (\mu_{\mR})$ which is the projection of the Figure~\ref{ElismAttracIFSPic}.}\label{ElismAttracGIFSPic}
\end{figure}

The ergodic theorem for projected Hutchinson measures from \cite{EO} claims that
$$
\lim _{N \rightarrow+\infty} \frac{1}{N} \sum_{n=0}^{N-1} f\left(x_{n}(a)\right)=\int_{X} f(t) d \nu'(t)
 $$
where $x_{n}(a)$ is the projection on the first coordinate and $a$ is a sequence of symbols chosen with probability one in the sequences of $\{1,2\}^{\mathbb{N}} $ according to the probabilities $p_1:=0.65$ and $p_2:=0.35$.

Consider $f(x)=x^2$: using the measure obtained by \texttt{GIFSMeasureDraw}($\mR$) for estimate $\int_{0}^{1} x^2 d\nu'(x)$, the projected measure, we obtain
$$\int_{0}^{1} x^2 d\nu'(x) \, \simeq 0.12177521930...$$

On the other hand, using the ergodic theorem where each $a_i$ in the sequence $a=(a_0,a_1,a_2, ...)$ is picked from a random i.i.d.\~ variable $I \in\{1,2\}$ with distribution $p_1:=0.65$ and $p_2:=0.35$, $x_0=0.5$ and $N=10.000$ we get
$$
 \frac{1}{10000} \sum_{n=0}^{10000-1} \left(x_{n}(a)\right)^2 \,\simeq 0.1228183842857\ldots
 $$
with an absolute error of $5\times 10^{-2}$. For a fractal computed on a resolution of $500\times500$ pixels the error is $3\times 10^{-4}$.

\subsection{IFS and GIFS with place dependent probabilities}
We consider variable probabilities, that is, each $p_j$ is as function of $x$, such as in \cite{Hut}, \cite[Theorem 2.1]{BDEG}, assuming average-contractiveness, \cite{Ste}, \cite{Ob1} and more recently \cite{GMM} for IFS and in \cite[Section 3]{Mi} for GIFS.

We notice that Lemma~\ref{net prob} is still valid under the variable probability hypothesis. Thus we just need to ensure that the respective Markov operator is Banach contractive to use Theorem~\ref{dfp}. We are not going to remake several straightforward computations to update the algorithms. We only update the necessary computation of the probabilities in each case (both in the bidimensional version):
\begin{itemize}
  \item In the algorithm \texttt{GIFSMeasureDraw}($\mS$) we replace $$\nu_{(x_1[i^1],..., x_d[i^d])}:=\nu_{(x_1[i^1],..., x_d[i^d])} +  p_{j} \nu_{y_0}\cdot...\cdot\nu_{y_{m-1}}$$ by
$$\nu_{(x_1[i^1],..., x_d[i^d])}:=\nu_{(x_1[i^1],..., x_d[i^d])} +  p_{j}(y_0,..,y_{m-1}) \nu_{y_0}\cdot...\cdot\nu_{y_{m-1}}$$
\end{itemize}

\begin{example}\label{ExMiculescu}
   In \cite[p. 146]{M1}, the author considers the one dimensional case $X=[0,1]$ and a GIFS $\mS=(X,(\phi_j)_{j=1}^2)$ where
   \[
   \left\{
     \begin{array}{ll}
       \phi_1(x,y)= &  \frac{x}{4}+\frac{y}{4} \\
       \phi_2(x,y)= &  \frac{x}{4}+\frac{y}{4} +\frac{1}{2}.
     \end{array}
   \right.
   \]
Given a function
$$
\alpha(t)=\left\{\begin{array}{ll}{1,} & {\text { if } t \in\left[0, \frac{1}{4}\right]} \\ {2-4 t,} & {\text { if } t \in\left[\frac{1}{4}, \frac{1}{2}\right]} \\ {0,} & {\text { if } t \in\left[\frac{1}{2}, 1\right]}\end{array}\right.
 $$
he considers the probabilities $p_{1}(x, y)= \frac{1}{33} \alpha(x) \alpha(y)$,
$$
p_{2}(x, y)=1-p_{1}(x, y)=1-\frac{1}{33} \alpha(x) \alpha(y),
 $$
and the GIFSpdp $\mS=(X,(\phi_j)_{j=1}^2, (p_j)_{j=1}^2)$.

Under this hypothesis, he verifies that $M_{\mathcal{S}}$ is a contraction and $\mu_{\mathcal{S}}$ is its only fixed point. More than that, he verifies that the attractor is $A_{\mathcal{S}}=[0,1] $ and $\supp \mu_{\mathcal{S}} \subseteq \left[0, \frac{1}{4}\right] \cup\left[\frac{1}{2}, 1\right]$.

From the previous discussion it is easy to see that our algorithm \texttt{GIFSMeasureDraw}($\mS$) can be used to get approximations of $\mu_{\mathcal{S}}$ as we can see in figure~\ref{Ex_Miculescu_Histo}

\begin{figure}[ht]
  \centering
  \includegraphics[width=3cm,frame]{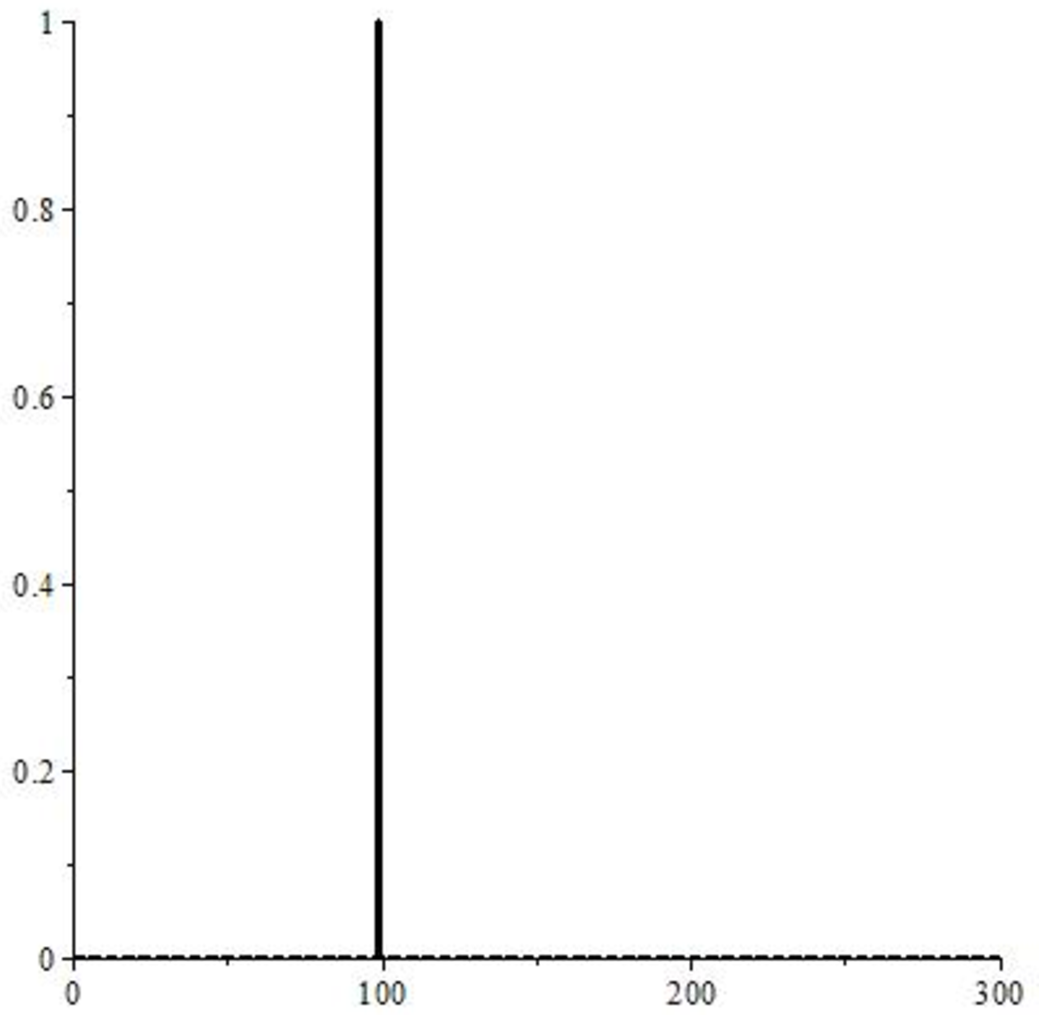}\;
  \includegraphics[width=3cm,frame]{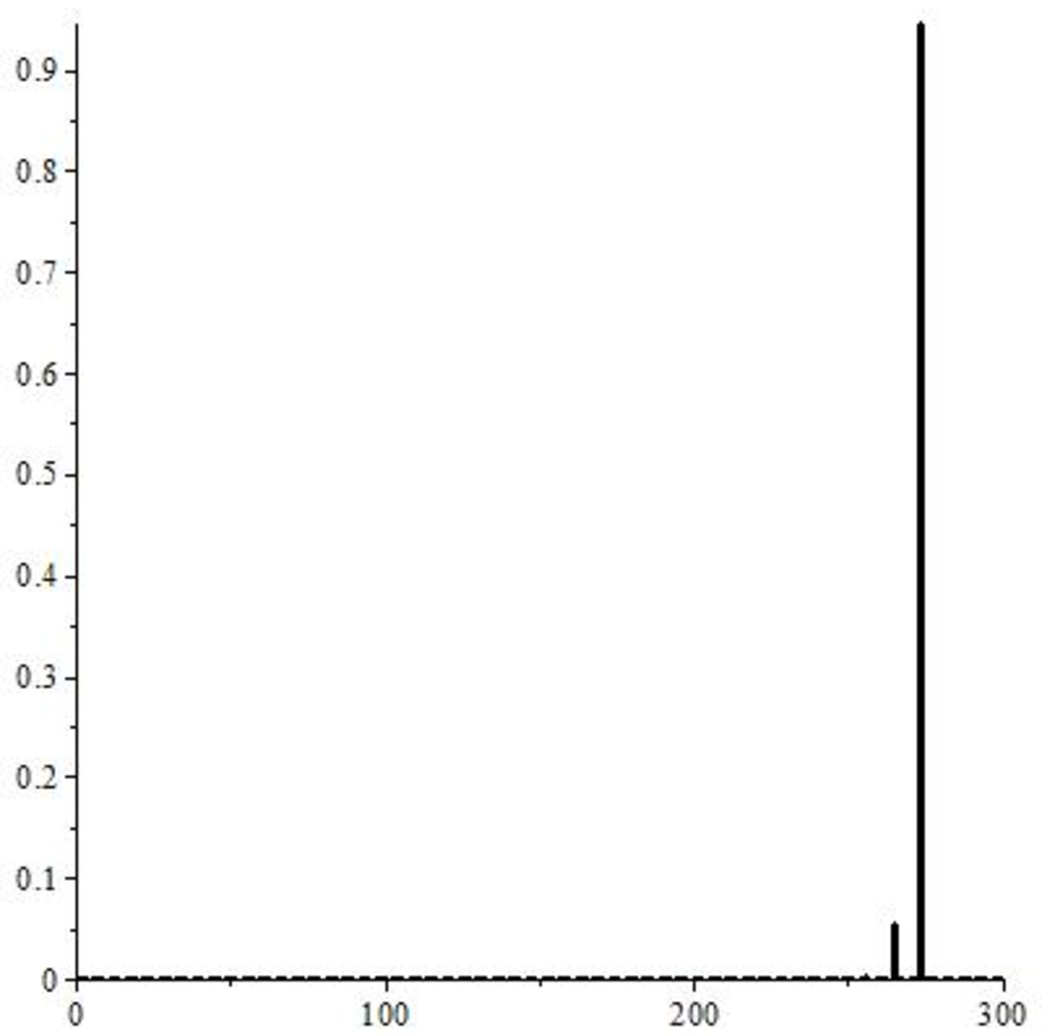}\;
  \includegraphics[width=3cm,frame]{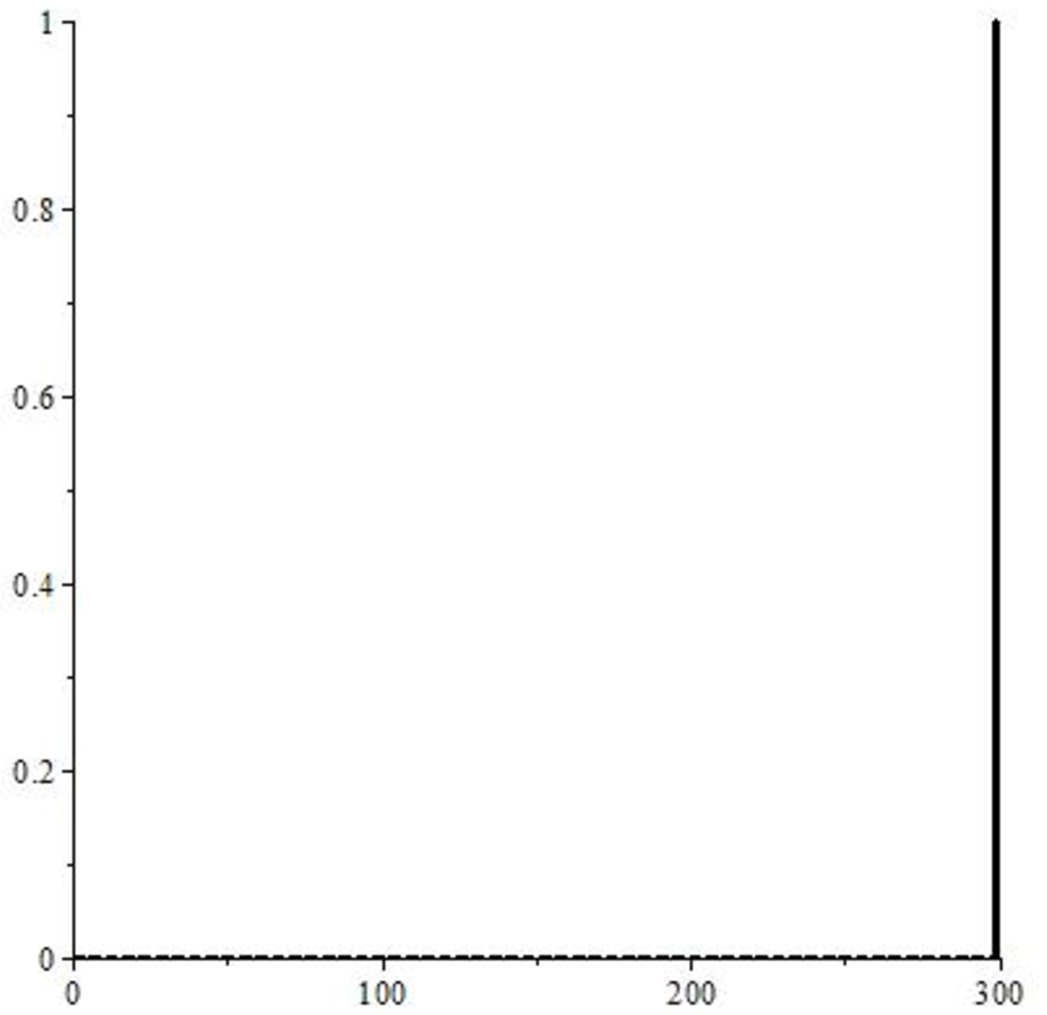}
  \caption{The output of the algorithm \texttt{GIFSMeasureDraw}($\mS$) after 1, 5 and 15 iterations, having a fairly high definition of $300$ pixels. }\label{Ex_Miculescu_Histo}
\end{figure}

The output suggests that $\mu_{\mathcal{S}}=\delta_{1}$! We can verify that by direct examination as follows.

Consider a continuous function $f:[0,1] \to \mathbb{R}$ then,
$$\int f(z) d M_{\mathcal{S}}(\delta_{1}\times \delta_{1})= \int_x\int_y p_{1}(x, y) f(\phi_1(x,y))+ p_{2}(x, y) f(\phi_1(x,y)) d \delta_{1}(x)d\delta_{1}(y)=$$
$$= p_{1}(1, 1) f(\phi_1(1,1))+ p_{2}(1, 1) f(\phi_1(1,1))= p_{1}(1, 1) f\left(\frac{1}{4}+\frac{1}{4}\right)+ p_{2}(1, 1) f\left(\frac{1}{4}+\frac{1}{4} +\frac{1}{2}\right)=$$
$$=f(1)= \int f(z) d \delta_{1}(z)),$$
because $p_{1}(1, 1)=0$. Therefore, $M_{\mathcal{S}}(\delta_{1}\times \delta_{1})=\delta_{1}$ meaning that $\mu_{\mathcal{S}}=\delta_{1}$.
\end{example}

\begin{example}\label{ex_conze}
   In this case we consider a negative case from \cite{CR}. In this case the authors consider the doubling map $T(x)=2x\mod 1$ on the interval and studies the solutions of the equation
   $$P_{u} f(x)=u\left(\frac{x}{2}\right) f\left(\frac{x}{2}\right)+u\left(\frac{x}{2}+\frac{1}{2}\right) f\left(\frac{x}{2}+\frac{1}{2}\right)
 $$ for a given potential $u:[0,1] \to \mathbb{R}$ and the $P_{u}$-harmonic functions, i. e. $P_{u}(f)=f$. Then they characterizes the dual solutions $P_{u}^*(\mu)=\mu$, which he calls invariant measures.

It follows that this problem is the same as the problem of finding the Hutchinson measure for the IFSpdp $\mS=(X,(\phi_j)_{j=1}^2, (p_j)_{j=1}^2)$ where
    \[
   \left\{
     \begin{array}{ll}
       \phi_1(x)= &  \frac{x}{2}  \\
       \phi_2(x)= &  \frac{x}{2} +\frac{1}{2}.
     \end{array}
   \right.
   \]
Given a function, say $u(t)=\cos^2(3 \pi t)$, there are considered probabilities $p_{1}(x )= u(\phi_1(x))=u(x/2)$ and $p_{2}(x )=u(\phi_2(x))= u(x/2 + 1/2)$.
\end{example}
\begin{figure}[ht!]
  \centering
  \begin{tabular}{cc}
  \includegraphics[width=3cm,frame]{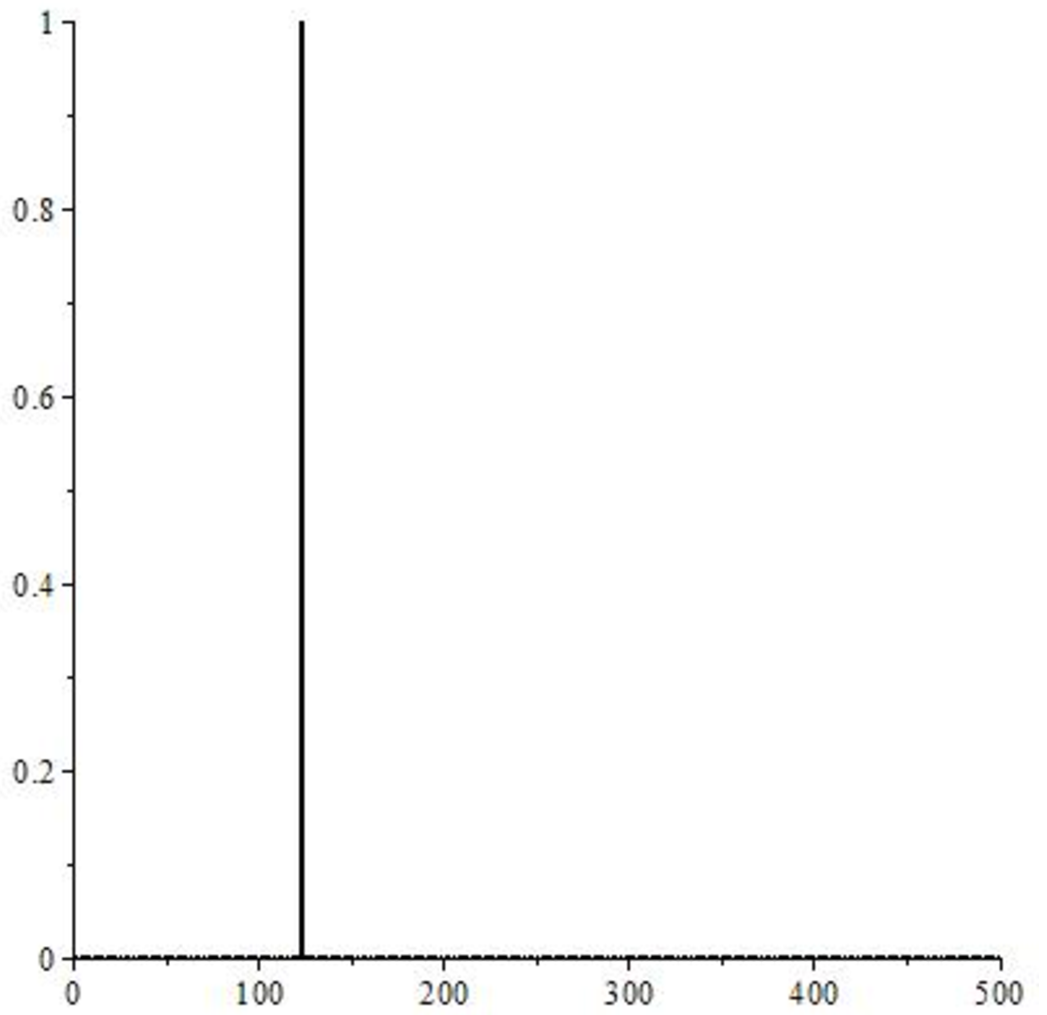}&
  \includegraphics[width=3cm,frame]{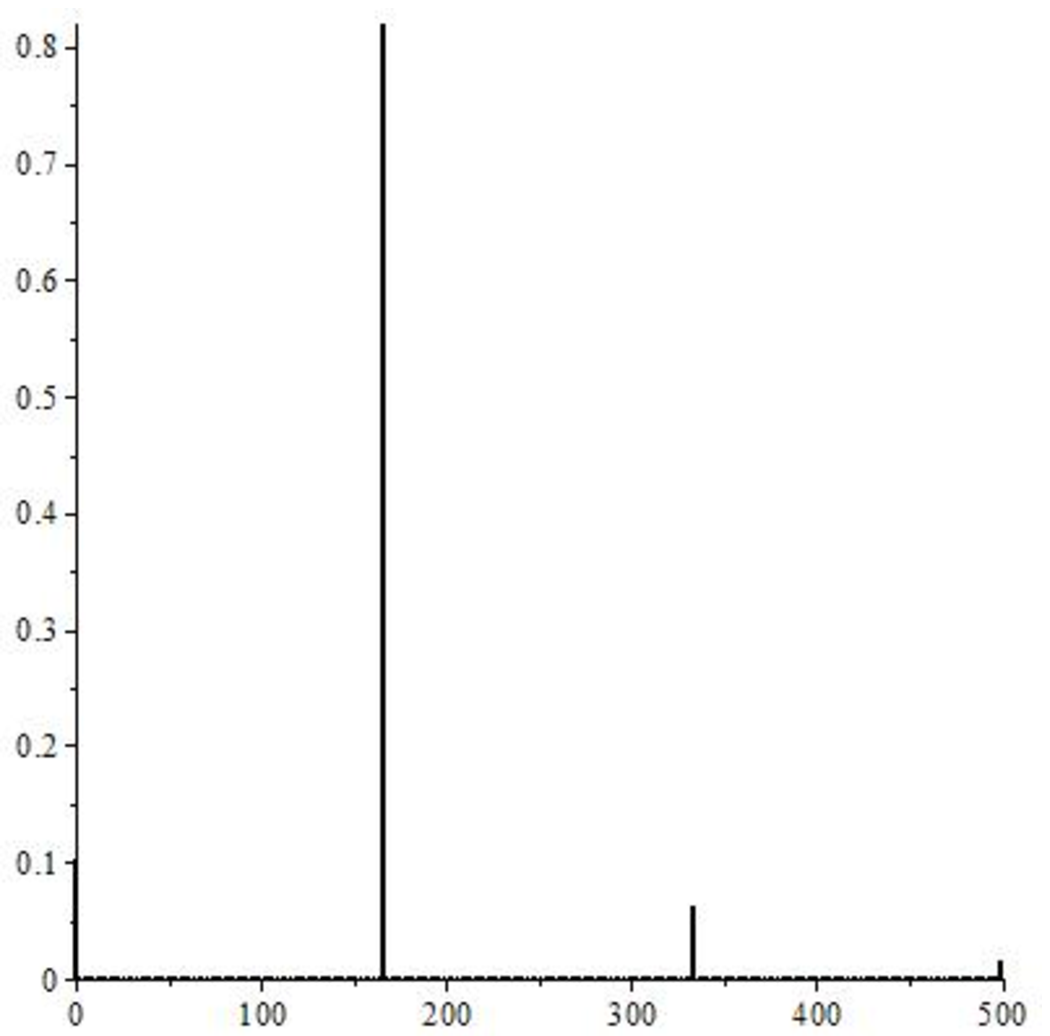}\\
  \includegraphics[width=3cm,frame]{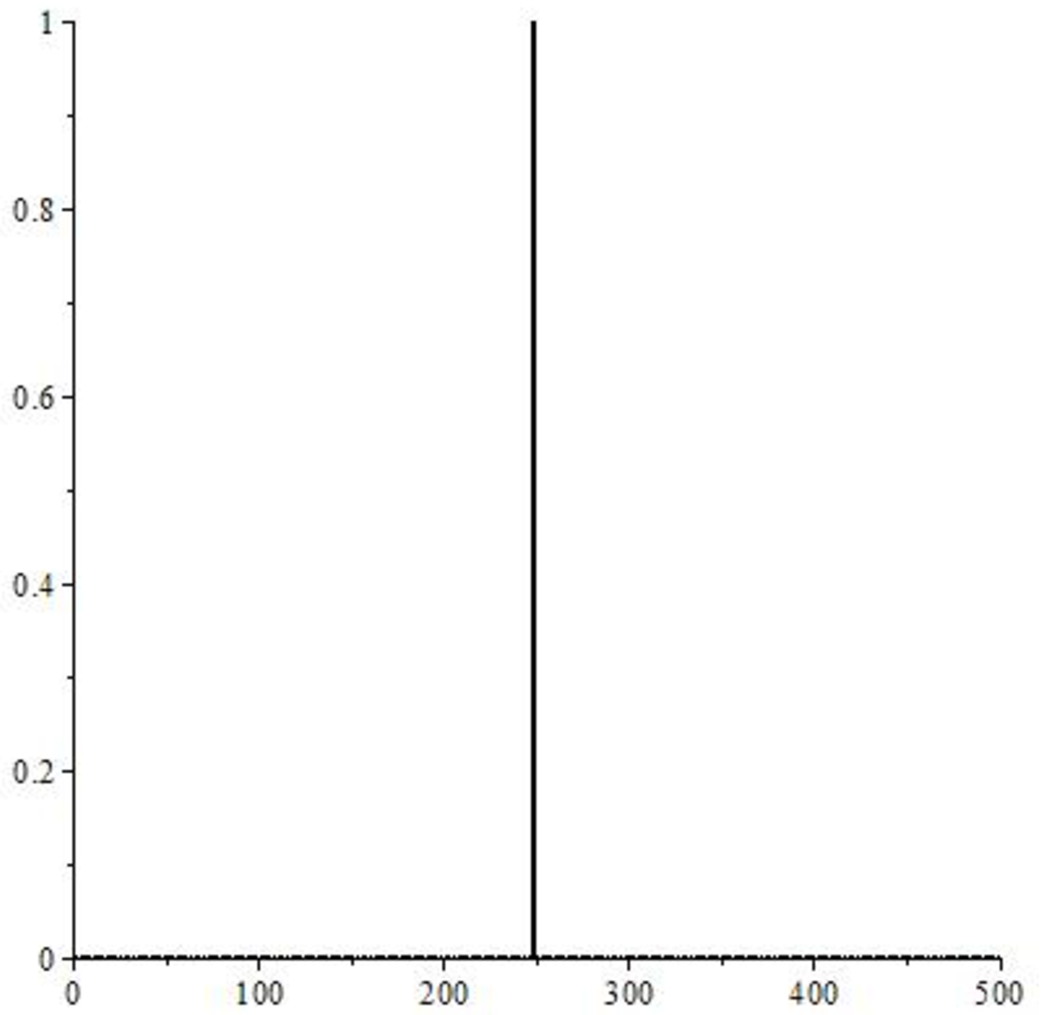}&
  \includegraphics[width=3cm,frame]{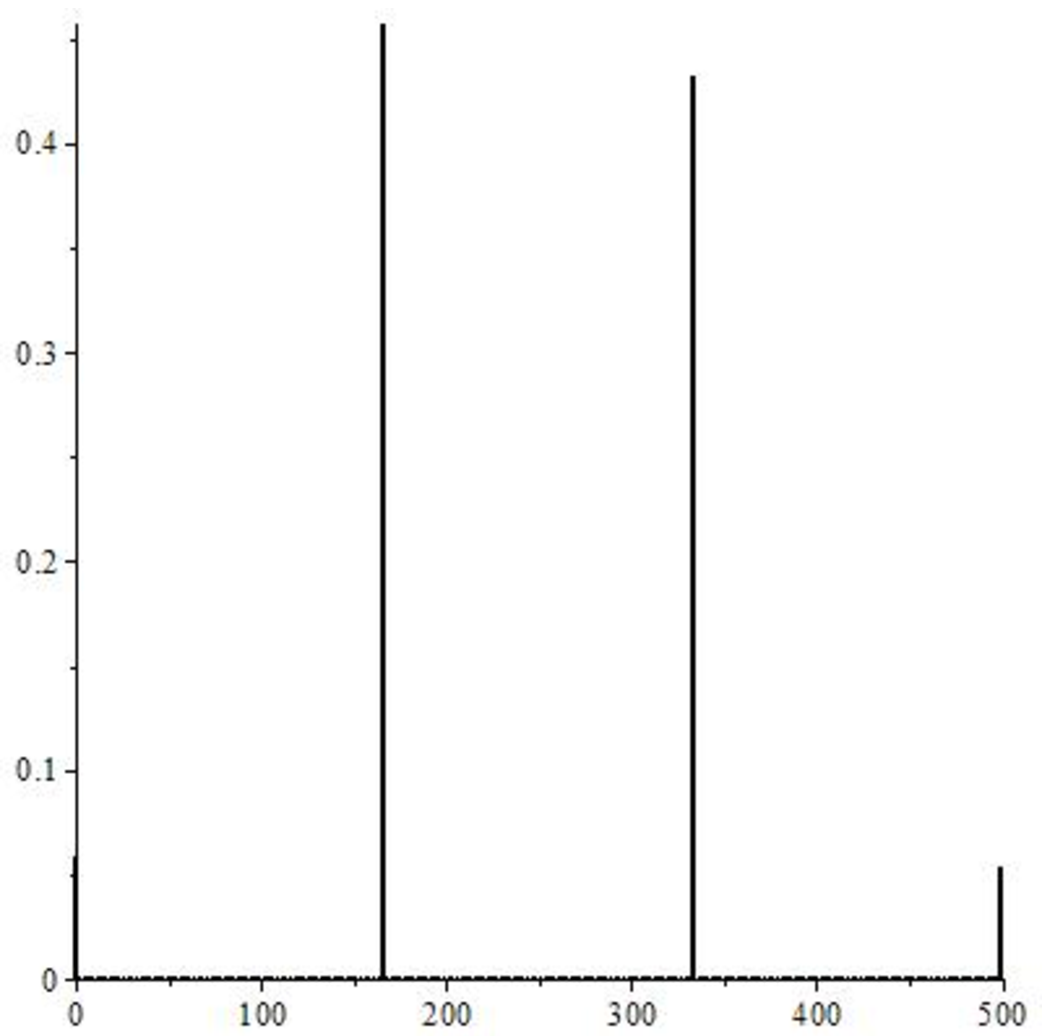}
  \end{tabular}
  \caption{The output of algorithm \texttt{GIFSMeasureDraw}($\mS$) after 1 and 100 iterations, having a fairly high definition of $500$ pixels. In the first row, after $100$ iterations the discrete measure is $\mu:= 0.1030321845 \, \delta_{0} + 0.8199775187 \, \delta_{\frac{1}{3}}+ 0.06209698542 \, \delta_{\frac{2}{3}}+ 0.01471437940 \, \delta_{1}$; in the second row, after $100$ iterations the discrete measure is $\mu:= 0.05792616463 \,  \delta_{0} + 0.4568641232 \, \delta_{\frac{1}{3}}+ 0.4317603436 \, \delta_{\frac{2}{3}}+ 0.05323785154 \, \delta_{1}$.}\label{conze_500_100_s4}
\end{figure}
In \cite[example 3 b]{CR}, the authors show that the invariant measures (Hutchinson measures for the IFSpdp) are the probabilities supported on the $T$-periodic cycles $\{0\}$, $\{1\}$ and $\{\frac{1}{3}, \frac{2}{3}\}$.  In that case our Markov operator cannot be a contraction because there is more than one Hutchinson measure; however, we can still use \texttt{GIFSMeasureDraw}($\mS$) to see if it converges to any measure. We show in Figure~\ref{conze_500_100_s4} a selection phenomenon: the iteration seems to converge to some combination $\mu:= a \delta_{0} + b \delta_{\frac{1}{3}}+c \delta_{\frac{2}{3}}+d \delta_{1}$ depending on the initial measure we choose.

\section*{Acknowledgments}
We are grateful to the referees whose careful reading improved the paper.

\end{document}